\documentclass[10pt]{article}

\usepackage{amsmath,amsthm,amssymb}
\usepackage{caption,graphicx,subfig,xcolor}
\usepackage[text={7in,9in},centering]{geometry}
\usepackage{setspace}
\usepackage[backref=none,hypertexnames=false,colorlinks=true,urlcolor=blue,linkcolor=blue,citecolor=blue]{hyperref}
\usepackage{ucs}
\usepackage[utf8x]{inputenc}
\usepackage{graphicx}
\usepackage{amsfonts}
\usepackage{dsfont}
\usepackage{amssymb}
\usepackage{amsmath}
\usepackage{amsthm}
\usepackage{enumerate}
\usepackage{stmaryrd}
\usepackage{fullpage}
\usepackage{ifthen}
\usepackage{epic}
\usepackage{authblk}
\usepackage{textcomp}
\usepackage{soul}
\usepackage[title]{appendix}

\usepackage[hypertexnames=false,colorlinks=true,linkcolor=blue,citecolor=blue]{hyperref}
\usepackage[numbers,comma,square,sort&compress]{natbib}

\makeatletter\@addtoreset{equation}{section}\makeatother

\newtheorem{thm}{Theorem}[section] 
\newtheorem{lem}[thm]{Lemma}  
\newtheorem{cor}[thm]{Corollary} 
\newtheorem{prop}[thm]{Proposition}  
\newtheorem{hyp}{Hypothesis}

\newtheoremstyle{named}{}{}{\itshape}{}{\bfseries}{.}{.5em}{\thmnote{#3's }#1}
\theoremstyle{named}

\theoremstyle{definition}
\newtheorem{rmk}{Remark}

\newcommand{\comment}[1]{}

\newcommand{\Kncrit}{K_{\mathrm{crit,n}}}
\newcommand{\Kcrit}{K_\mathrm{crit}}

\title{Capturing the critical coupling of large random Kuramoto networks with graphons}

\author[1]{Jason Bramburger}
\author[2,3]{Matt Holzer}
\affil[1]{\small Department of Mathematics and Statistics, Concordia University, Montr\'eal, QC, Canada}
\affil[2]{\small Department of Mathematical Sciences, George Mason University, Fairfax, VA, USA }
\affil[3]{\small Center for Mathematics and Artificial Intelligence (CMAI), George Mason University, Fairfax, VA, USA}

\date{}

\begin{document}

\maketitle

\begin{abstract}
Collective oscillations and patterns of synchrony have long fascinated researchers in the applied sciences, particularly due to their far-reaching importance in chemistry, physics, and biology. The Kuramoto model has emerged as a prototypical mathematical equation to understand synchronization in coupled oscillators, allowing one to study the effect of different frequency distributions and connection networks between oscillators. In this work we provide a framework for determining both the emergence and the persistence of synchronous solutions to Kuramoto models on large random networks and with random frequencies. This is achieved by appealing the theory of graphons to analyze a  continuum model coming in the form of an infinite oscillator limit which provides a single master equation for studying random Kuramoto models. We show that bifurcations to synchrony and hyperbolic synchrony patterns in the  continuum model can also be found in related random Kuramoto networks for large numbers of oscillators. We further provide a detailed application of our results to oscillators arranged on Erd\H{o}s--R\'enyi random networks, for which we further identify that not all bifurcations to synchrony emerge through simple co-dimension one bifurcations.    Finally, the method of proof can be applied to interacting particle systems beyond the Kuramoto model to establish the persistence of saddle-node bifurcations.
\end{abstract}

\section{Introduction}

Many processes throughout the applied sciences can be modeled as sets of interacting periodic processes, particularly in neuroscience \cite{bick2020understanding}. A major focus for mathematical investigation of these networks is to identify whether or not these oscillations fall into global patterns of synchrony. Synchrony of neuronal oscillators governs many cognitive tasks and functions \cite{uhlhaas2009neural,singer1999neuronal}, including playing a critical role in memory formation \cite{axmacher2006memory,jutras2010synchronous}. Synchrony is also important for the functioning of power grid networks \cite{motter2013spontaneous,rohden2012self}. Alternatively, certain patterns of synchrony in neuronal networks have been associated with epilepsy and Parkinson's disease \cite{hammond2007pathological,lehnertz2009synchronization}. Thus, in the study of networks of oscillating processes, it is often important to determine whether synchrony can occur and what form it takes. 

Synchronization in the mathematical literature is often studied in the Kuramoto phase model
\cite{kuramoto1975self,kuramoto1984chemical,acebron2005kuramoto}
\begin{equation}\label{Kuramoto}
    \dot\theta_j = \omega_j + \frac{K}{n}\sum_{k = 1}^n A_{j,k}\sin(\theta_k - \theta_j), \qquad j = 1,\dots, n.
\end{equation}
Here each $\theta_j \in S^1$ represents the relative phase of oscillator $j$, with $\omega_j \in [-1,1]$ being their intrinsic natural frequency, and $K \geq 0$ the strength of coupling between oscillators. The  symmetric matrix $[A_{j,k}]_{j,k = 1}^n$ is an  undirected graph adjacency matrix encoding the network structure so that $A_{j,k} >0$ denotes a connection between oscillators $j$ and $k$, while $A_{j,k} = 0$ represents the absence of one. Synchronization in \eqref{Kuramoto} occurs when there exists solutions satisfying $\dot \theta_j(t) = \dot\theta_k(t)$ for all $j,k = 1,\dots n$, meaning that oscillators evolve with the same velocity, differing only by an initial phase offset, termed a {\em phase-lag}. When $K = 0$ all oscillators act independently and no synchronization occurs, while in the limit $K\to \infty$ many synchronized states exist with $|\theta_j(t) - \theta_k(t)| \in \{0,\pi\}$. Thus, one concludes that there exists an intermediary coupling strength $K = \Kcrit > 0$, termed the {\em critical coupling} \cite{ling2020critical,dorfler2011critical,strogatz2000kuramoto}, at which the existence of synchronized states first appears in \eqref{Kuramoto}. For complete graphs, i.e. $A_{j,k} = 1$ for all $(j,k)$, D\"orfler and Bullo \cite{dorfler2011critical} survey much of the current landscape, while further providing upper and lower bounds on $\Kcrit$. However, these bounds remain at a finite distance from each other for all $n \geq 1$. Related work in \cite{ling2020critical} provides upper bounds on $\Kcrit$ for dense networks, while \cite{dorfler2013synchronization} bounds the critical coupling from above based upon the specifics of the network topology and the intrinsic frequencies.

In attempting to understand the onset of synchrony in \eqref{Kuramoto} with $n \gg 1$, one may formally pass to a limiting  continuum model
\begin{equation}\label{KuramotoGraphon}
    \frac{\partial \theta}{\partial t} = \Omega(x) + K\int_0^1 W(x,y)\sin(\theta(y,t) - \theta(x,t))\mathrm{d}y,
\end{equation}
for a continuous function $\Omega:[0,1] \to [-1,1]$ and a kernel $W:[0,1]^2 \to [0,1]$.   Pioneering work in this direction was done by Ermentrout \cite{ermentrout1985synchronization} who used \eqref{KuramotoGraphon} with $W \equiv 1$ to estimate $\Kcrit$ with all-to-all coupling and random frequencies in \eqref{Kuramoto} when $n \gg 1$. A second use for \eqref{KuramotoGraphon} is to study {\em identically coupled oscillators}, i.e. $\omega_j = \omega_k$ in \eqref{Kuramoto} for all $j,k = 1,\dots,n$, by having $\Omega$ be a constant function and identifying both the existence and stability of synchronous patterns with different network topologies. Early work in this direction comes from \cite{strogatz2000kuramoto}, which arranges the oscillators in a ring to observe patterns of synchrony whose phase-lags increase monotonically around it. More recent investigations have employed graphons \cite{lovasz12} to use \eqref{KuramotoGraphon} to capture patterns of synchrony in \eqref{Kuramoto} over random networks \cite{medvedev2014small,bramburger2024persistence,nagpal2024synchronization}. 

In this paper we leverage graphon theory to capture the existence of synchronous solutions in \eqref{Kuramoto} on random networks with random frequencies. This work extends contributions such as \cite{ermentrout1985synchronization,chiba2018bifurcations} that use \eqref{KuramotoGraphon} to study large random Kuramoto models to justify findings for finite $n \gg 1$ models of the form \eqref{Kuramoto}. Moreover, this work contributes to the growing literature on synchronization of Kuramoto models on random graphs \cite{bick2024dynamical,ling2019landscape,abdalla2022expander,kassabov2022global,medvedev2014nonlinear,nagpal2024synchronization}, while extending these studies to allow for random frequencies as well. In particular, we show that \eqref{KuramotoGraphon} provides a single master equation to identify the critical coupling in, which asymptotically estimates the critical coupling in classes of random Kuramoto models \eqref{Kuramoto}. 

The appeal to graphon theory allows one to think of groups of networks over different numbers of vertices as belonging to the same family, represented by the limiting function (the graphon) $W$ in \eqref{KuramotoGraphon}. Thus, for all $n \gg 1$ and networks with adjacency matrices $[A_{j,k}]_{j,k = 1}^n$ belonging to the same family as a graphon $W$, our analysis shows that with frequencies $\{\omega_j\}_{j = 1}^n$ drawn independently from a distribution on $[-1,1]$, with we have that  asymptotically almost surely -- that is, with probability tending to one as $n$ tends to infinity:    
\begin{enumerate}
    \item Hyperbolic synchronous solutions in \eqref{KuramotoGraphon} whose phase-lags vary continuously over $x \in [0,1]$ at a fixed $K$ lead to hyperbolic synchronous solutions to \eqref{Kuramoto} at the same coupling value $K$,
    \item Saddle-node bifurcations at $K = \Kcrit$ of  synchronous solutions in \eqref{KuramotoGraphon} with continuous phase-lags persist as saddle-node bifurcations of synchronous solutions in \eqref{Kuramoto} at some $K = \Kncrit \approx \Kcrit$. 
\end{enumerate}
 These results suppose that the graphon $W$ satisfies a technical condition, see (\ref{Wcompactness}) below, so that the integral operator in (\ref{KuramotoGraphon}) maps $C[0,1]$ to $C[0,1]$.

Thus, in many cases, these results fully capture the emergence and persistence of synchronous solutions in \eqref{Kuramoto} for $n\gg 1$ using the single master equation \eqref{KuramotoGraphon}. However, as we also show in this work, the emergence of synchronous solutions in \eqref{KuramotoGraphon} is not always attributed to a saddle-node bifurcation. In particular, we prove that for Erd\H{o}s--R\'enyi random networks and certain distributions of random frequencies, the emergence of synchronous solutions in \eqref{KuramotoGraphon} comes from a bifurcation involving the essential spectrum.   In this scenario, we are able to apply the persistence results of \cite{bramburger2024persistence} to show that for any $K>\Kcrit$ there exists $n$ sufficiently large so that stable synchronous solutions exist in (\ref{Kuramoto}) for that choice of coupling constant.  This provides some rigorous justification of $\Kcrit$ as being a relevant predictor for the onset of synchrony in the case of a bifurcation owing to essential spectrum.   While we are unable to rigorously prove the existence of a saddle-node bifurcation to synchrony in the case of essential spectrum, we do provide numerical observations that indicate that even though our hypotheses do not hold, our results do. This leads one to at least conjecture that similar results could be obtained for more complex bifurcation scenarios only present in infinite-dimensional system.

Before proceeding, we note that analysis of \eqref{Kuramoto} for large numbers of oscillators with all-to-all coupling is often studied in a mean-field limit, distinct from \eqref{KuramotoGraphon}, in which the evolution of probability densities describing the oscillators is studied; see \cite{kuramoto1975self,sakaguchi88,strogatz2000kuramoto,ott08}. Development of an analogous model incorporating network interactions using graphon theory was obtained in \cite{chiba2018mean}. The  probability density paradigm is advantageous as it allows for the convenient study of various time dependent solutions to \eqref{Kuramoto} including the bifurcation of the incoherent state \cite{chiba15,dietert18}.  That being said, we find \eqref{KuramotoGraphon} a more suitable tool for the mathematical treatment of the existence and stability of synchronized steady states in \eqref{Kuramoto} for classes of random networks with sufficiently large numbers of oscillators. 

Beyond the applications considered here we also note that the mathematical methods developed in this manuscript can be viewed as an extension the persistence results obtained by the authors in \cite{bramburger2024persistence} to include the persistence of saddle-node bifurcations.  The techniques developed here could be applied to other systems of interacting processes; see \cite{bramburger2024persistence} for other possible application areas. 

This paper is organized as follows. In Section~\ref{sec:Graphons} we provide the relevant background theory for graphons. Then, in Section~\ref{sec:Results} we provide our hypotheses and precise statements of the main results summarized informally above. Section~\ref{sec:ApplicationER} turns to applying these results to random Kuramoto models posed on Erd\H{o}s--R\'enyi networks, including proving that we can either have saddle-node bifurcations to synchronous solutions in \eqref{KuramotoGraphon} or the more complex essential spectrum bifurcations. The proofs of our results are left to Sections~\ref{sec:CentreManifoldProof} and \ref{sec:PersistenceProof}. We conclude in Section~\ref{sec:Discussion} with a discussion of our findings and numerous avenues of potential future research.

\section{Graphons}\label{sec:Graphons}

To obtain the results in this paper we appeal to the theory of graphons; see \cite{borgs11,lovasz12}. A graphon is a symmetric function $W:[0,1]^2 \to [0,1]$ that can be used to represent the edge weight $W(x,y)$ of a graph with infinitely many vertices $x,y\in[0,1]$ . One can see the presence of the graphon in the  continuum model \eqref{KuramotoGraphon}, here functioning as an integral kernel. Boundedness of graphons implies that they always belong to $L^p([0,1]^2)$ for all $p \in [1,\infty]$, however a more natural measure of distance on the set of graphons is given by the cut norm,
\begin{equation}
    \|W\|_\square = \sup_{S,T \subseteq [0,1]} \bigg|\int_{S\times T} W(x,y)\ \mathrm{d}x\mathrm{d}y\bigg|,
\end{equation}
where the supremum is taken over all measurable subsets $S,T$ of $[0,1]$. There are many other equivalent forms of the cut norm (see \cite[Appendix~E]{janson2010graphons}), all of which treat the graphon not as a function as the $p$-norms would, but as an integral kernel. Indeed, \cite[Lemma~E.6]{janson2010graphons} shows that the integral operator $T_W:L^p \to L^q$ acting by
\begin{equation}
    [T_Wf](x) = \int_0^1 W(x,y)f(y)\mathrm{d}y 
\end{equation}
has operator norm bounded as $\|W\|_\square \leq \|T_W\|_{p\to q} \leq 2\sqrt{2}\|W\|_\square^{\min\{1 - 1/p,1/q\}}$ for all $p,q\in[1,\infty]$. Moreover, convergence in the cut norm does not necessarily imply convergence in the $p$-norms, while the converse is always true since $\|W\|_\square \leq \|W\|_p$ for every graphon $W$.

Graphons and the cut norm find significant application as a rule for generating families of finite graphs through sampling. Precisely, let $ \{x_1,x_2,\dots,x_n\}$ be an ordered $n$-tuple of independent uniform random points drawn from $[0,1]$. We describe two different random graphs generated from a single graphon $W$:
\begin{enumerate}
    \item Let $\mathbb{H}(n,W)$ denote the weighted graph with vertices $\{1,2,\dots,n\}$ and edge weights $W(x_j,x_k)$ between vertices $j$ and $k$, $j \neq k$. Loops are given edge weight 0.
    \item Let $\mathbb{G}(n,W)$ denote the simple graph with vertices $\{1,2,\dots,n\}$ which are connected with an edge of weight 1 with probability $W(x_j,x_k)$, $j \neq k$. Loops again have an edge weight of 0. 
\end{enumerate}
To compare these random graphs with their generating graphon we consider a step graphon generated by a graph $G$. To do this, partition $[0,1]$ into $n$ disjoint intervals of equal length $I_1^n,I_2^n,\dots,I_n^n$, so that the step function $W_G:[0,1]^2 \to [0,1]$ takes the value of the edge weight between vertices $j$ and $k$ of $G$ for all $x \in I_j^n$ and $y \in I_k^n$. Lemma~10.16 of \cite{lovasz12} gives that the bounds 
\begin{equation} \label{eq:cutnormconv}
    \|W_{\mathbb{H}(n,W)} - W\|_\square \leq \frac{20}{\sqrt{\log(n)}}, \qquad \|W_{\mathbb{G}(n,W)} - W\|_\square \leq \frac{22}{\sqrt{\log(n)}}
\end{equation}
hold with probability at least $1 - \mathrm{exp}(-n/(2\log(n)))$ for all $n \geq 1$. Thus, we achieve convergence in probability of the families of random graphs generated by a graphon $W$.

\begin{rmk}\label{rmk:Graphon}
    Throughout this work we will always consider the sample points $\{x_1,x_2,\dots,x_n\}$ to be drawn independently from the uniform distribution on $[0,1]$. However, there are many other ways of generating these sequences to achieve almost sure convergence in the cut norm \cite[Lemma~11.33]{lovasz12}. For example, one may fix $x_j = j/n$ for $j = 1,\dots,n$ or drawn each $x_j$ from the uniform distribution on $[(j - 1)/n,j/n]$. Effectively, any sequence that is a good set for numerical integration, meaning 
    \begin{equation}
        \int_0^1 f(x)\mathrm{d}x \sim \frac{1}{n}\sum_{j = 1}^n f(x_j),
    \end{equation}
    with small error for Riemann integrable $f$, will do.
\end{rmk}

We will also consider the graphon analogue of the degree of a vertex $x \in [0,1]$, given by
\begin{equation}
    d_W(x) := \int_0^1 W(x,y)\mathrm{d}y. 
\end{equation}
The above is simply the continuum analogue of the degree of a vertex in a graph normalized by the number of vertices $n$. Along with the cut norm, another measure of convergence of graphs derived from graphons to their generating graphon is through their degree functions. In particular, \cite[Lemma~I]{garin2024corrections}\footnote{This is an improved version of a result found in \cite{vizuete2021laplacian}.} proves that for any graphon $W$, with probability $1 - \nu$ we have
\begin{equation}\label{DegreeConv}
    \|d_{W_{\mathbb{H}(n,W)}} - d_{W_{\mathbb{G}(n,W)}}\|_\infty = \sup_{x \in [0,1]} |d_{W_{\mathbb{H}(n,W)}}(x) - d_{W_{\mathbb{G}(n,W)}}(x)| \leq \sqrt{\frac{\log(2n/\nu)}{n}},
\end{equation}
thus allowing for a comparison between the degrees of the weighted and simple graphs derived from $W$. Moreover, in many cases one can show that $\|d_{W_{\mathbb{H}(n,W)}} - d_W\|_\infty$ converges to 0 in probability, including for {\em ring graphons}, i.e. $W(x,y) = W(|x - y|)$, since the degree functions are constant \cite{bramburger2023pattern}. Thus, combining the above allows one to show that at least for the case of ring graphons we can find $\|d_{W_{\mathbb{G}(n,W)}} - d_W\|_\infty \to 0$ in probability as well. Our main result is achieved by assuming degree convergence, while all demonstrations use graphons for which we know this holds.

\section{Main Results}\label{sec:Results}

We now leverage the theory of graphons from the previous section to provide our main results. In particular, we will prove that \eqref{KuramotoGraphon} is a single master equation for analyzing the existence, stability, and in some cases the onset of synchronous solutions for random Kuramoto models \eqref{Kuramoto} with large numbers of oscillators $n \gg 1$. The advantage is that, under the mild assumptions that follow, one can provide explicit results for infinitely many Kuramoto models with random frequencies and/or posed on random networks.

We will begin with the following assumption that allows us to side-step the probabilistic framework of random graphs and graphons laid out previously. For any $n \geq 1$, subdivide the interval $[0,1]$ into $n$ disjoint intervals $I_j^n = [(j-1)/n,j/n)$ of equal length and let $W_n:[0,1]^2 \to [0,1]$ be a step graphon on the partition $\{I_j^n \times I_k^n\}_{j,k = 1}^n$ of $[0,1]^2$. Similarly, let $\Omega_n:[0,1] \to [-1,1]$ be a step function on $\{I_j^n\}_{j = 1}^n$. The values on the steps of $\Omega_n$ will be the frequences $\{\omega_j\}_{j = 1}^n$ for the discrete model \eqref{Kuramoto} in what follows. We assume the following.

\begin{hyp}\label{hyp:ConvergenceSequences}
    There exists sequences $\{\Omega_n\}_{n = 1}^\infty$ and $\{W_n\}_{n = 1}^\infty$, along with a continuous function $\Omega:[0,1] \to [-1,1]$ and a graphon $W:[0,1]^2 \to [0,1]$ so that
    \[
        \lim_{n \to \infty}\|\Omega_n - \Omega\|_\infty = 0, \quad \lim_{n \to \infty}\|W_n - W\|_\square = 0, \quad \lim_{n \to \infty}\|d_{W_n} - d_W\|_\infty = 0. 
    \] 
    Furthermore, the graphon $W$ is such that for every $\varepsilon > 0$, there exists a $\delta > 0$ so that for all $x_0 \in [0,1]$ we have
    \begin{equation}\label{Wcompactness}
        \int_0^1 |W(x,y) - W(x_0,y)|\mathrm{d}y < \varepsilon
    \end{equation}
    when $|x - x_0| < \delta$ and $x \in [0,1]$.
\end{hyp}

\begin{rmk} Hypothesis~\ref{hyp:ConvergenceSequences} is introduced to avoid attaching probabilistic qualifiers to most rigorous statements made in this paper. The cut norm portion of the hypothesis holds almost surely by taking $W_n = W_{\mathbb{H}(n,W)}$ or $W_n = W_{\mathbb{G}(n,W)}$ for all $n \geq 1$.  Indeed, estimate (\ref{eq:cutnormconv}) implies that for any $\varepsilon>0$ the probability that $\|W_n - W\|_\square\geq \varepsilon$ for an infinite number of $n$ is zero by the Borel-Cantelli Lemma; see \cite[Theorem 4.3]{billingsley}. Remark~\ref{rmk:Graphon} indicates that many other such sequences exist.  Degree convergence is slightly more involved, but a similar argument and the discussion following (\ref{DegreeConv})  implies that degree convergence holds almost surely for ring-graphons which are degree constant.  We prove that $\|\Omega_n-\Omega\|_\infty\to 0$ almost surely in Lemma~\ref{lem:OmegaConv} below.

We emphasize that the cut norm and degree convergence portions of Hypothesis~\ref{hyp:ConvergenceSequences} holds almost surely for sequences of Erd\"os R\'eyni random graphs which constitute the main object of study in \S~\ref{sec:ApplicationER}.
\end{rmk}

The condition \eqref{Wcompactness} on $W$  guarantees that $T_W$ maps $C[0,1]$ to itself and will further be employed to prove compactness of the operator $T_W$ on the space $C([0,1])$ and is necessary to our results. It can be easily verified in the case of continuous $W$, while \cite[Appendix~A]{bramburger2024persistence} proves that it holds for piecewise continuous ring graphons. We now provide the following lemma that can be used to confirm the remaining portion of Hypothesis~\ref{hyp:ConvergenceSequences}, with the proof left to Appendix~\ref{app:OmegaConv}.

\begin{lem}\label{lem:OmegaConv} 
Let $\Omega:[0,1] \to [-1,1]$ be a continuous function and for each $n \geq 1$ let $ \{x_1,x_2,\dots,x_n\}$ be an ordered $n$-tuple of independent uniform random points drawn from $[0,1]$. If the step function $\Omega_n$ is assigned the value $\Omega_n(x) = \Omega(x_j)$ for all $x \in I_j^n$ and $j = 1,\dots, n$ then  
\[
    \lim_{n \to \infty}\|\Omega_n - \Omega\|_\infty = 0
\]
almost surely.
\end{lem}

We now seek to quantify synchronous solutions of both \eqref{Kuramoto} and \eqref{KuramotoGraphon} as roots of appropriate functions. A phase-locked solution to \eqref{KuramotoGraphon} takes the form $\theta(x,t) = \overline{\Omega}t + u(x)$, with $\overline{\Omega}\in\mathbb{R}$ and $u:[0,1] \to \mathbb{R}$, and solves $F(u,\overline{\Omega},K) = 0$ where
\begin{equation}\label{FFunction}
    F(u,\overline{\Omega},K) = \Omega(x) - \overline{\Omega} + K\int_0^1 W(x,y)\sin(u(y) - u(x))\mathrm{d}y,     
\end{equation}
where the graphon $W$ is fixed according to Hypothesis~\ref{hyp:ConvergenceSequences} above. One can find that any solution of $F(u,\overline{\Omega},K) = 0$ must have $\overline{\Omega} = \int_0^1 \Omega(x)\mathrm{d}x$. For sequences of step functions $\{\Omega_n\}_{n = 1}^\infty$ and $\{W_n\}_{n = 1}^\infty$ satisfying Hypothesis~\ref{hyp:ConvergenceSequences} we will further define the step version of \eqref{FFunction} 
\begin{equation}\label{FnFunction}
    F_n(u,\omega_n^*,K) = \Omega_n(x) - \omega_n^* + K\int_0^1 W_n(x,y)\sin(u(y) - u(x))\mathrm{d}y,   
\end{equation}
for each $n \geq 1$. If we further restrict $u(x)$ to be a step function on the same steps as $\Omega_n$ and $W_n$, solving $F_n = 0$ is equivalent to a finite-dimensional problem. This is because only the values ${\bf u} = \{u_j\}_{j = 1}^n\in\mathbb{R}^n$ on each step need to be found, thus solving $F_n = 0$ with $u$ restricted to a step function is equivalent to solving $[G_n({\bf u},\omega_n^*,K)]_j = 0$ for each $j = 1,\dots,n$, where
\begin{equation}\label{GFunction}
    [G_n({\bf u},\omega_n^*,K)]_j = \omega_j - \omega_n^* + \frac{K}{n}\sum_{k = 1}^n A_{j,k}\sin(u_k - u_j), \qquad j = 1,\dots, n,     
\end{equation}
with $\omega_j = \Omega(x_j)$ and $A_{j,k} = W(x_j,x_k)$. Notice that roots of \eqref{GFunction} lie in one-to-one correspondence with synchronous solutions of \eqref{Kuramoto} as $\theta_j(t) = \omega_n^*t + u_j$ for each $j = 1,\dots,n$. 

Since both \eqref{FFunction} and \eqref{GFunction} exhibit a translational invariance in the phase variables, it follows that solutions (if they exist) are never unique. To eliminate this redundancy, when considering solutions of equation (\ref{FFunction})  we will  restrict to the space of mean-zero continuous functions, denoted
\begin{equation}\label{MeanZeroSet}
    X = \bigg\{u \in C([0,1]):\ \int_0^1 u(x)\mathrm{d}x = 0\bigg\},
\end{equation}
and equipped with the supremum norm $\|\cdot\|_\infty$. We now present our first result.

\begin{thm}\label{thm:Persistence}
    Assume Hypothesis~\ref{hyp:ConvergenceSequences} and suppose that for a fixed $K > 0$ there exists $u^* \in X$ satisfying $F(u^*,\overline\Omega,K) = 0$ with $\overline\Omega = \int_0^1 \Omega(x)\mathrm{d}x$. Suppose further that the linearization of $F$ about $u^*$ on $X$, denoted $DF(u^*,\overline\Omega,K):X \to X$, is invertible with bounded inverse. Then, for each $\rho > 0$, there exists an $N \geq 1$ such that for all $n\geq N$ there is a vector $({\bf u}_n^*,\omega^*_n) \in \mathbb{R}^n\times \mathbb{R}$ satisfying $G_n({\bf u}_n^*,\omega^*_n,K) = 0$ and $\max\{\|u_n^* - u^*\|_\infty,|\overline\Omega - \omega^*_n|\} < \rho$, where $u_n^*\in X$ is the step function representation of ${\bf u}_n^*$ over $\{I_j^n\}_{j = 1}^n$. Furthermore, if $u^*$ is a stable solution of \eqref{FFunction}, then there exists an $M \geq 1$ so that for all $n\geq \max\{N,M\}$ the solution $({\bf u}_n^*,\omega^*_n)$ of \eqref{GFunction} is stable as well.     
\end{thm}

Theorem~\ref{thm:Persistence} describes the persistence of {\em hyperbolic} synchronous solutions to the graphon model \eqref{FFunction} for fixed values of the coupling constant $K$. In particular, it states that if a hyperbolic solution of \eqref{FFunction} exists, then with $n$ taken sufficiently large, a similar synchronous solution can be found in \eqref{GFunction} for the same coupling $K$. However, Theorem~\ref{thm:Persistence} does not deal with the onset of synchronization from a saddle-node bifurcation by varying the coupling coefficient $K$ since, by definition, hyperbolicity is violated at such a bifurcation point. Thus, our second result will prove the persistence of such saddle-node bifurcations, under the following hypothesis. 

\begin{hyp}\label{hyp:SNgraphon}
    There exists a $\Kcrit>0$ and $u^* \in X$ such that $F(u^*,\overline\Omega,\Kcrit) = 0$ with $\overline\Omega = \int_0^1 \Omega(x)\mathrm{d}x$ and the following properties 
    \begin{enumerate}
        \item[i)]The linear operator $DF(u^*,\overline\Omega,\Kcrit):X \to X$ has stable spectrum with the exception of a zero eigenvalue with algebraic and geometric multiplicity one. Precisely, there exists a $v^*\in X$ normalized so that $\int_0^1 [v^*(x)]^2\mathrm{d}x=1$ such that 
    \[ 
        \mathrm{Ker}(DF(u^*,\overline\Omega,\Kcrit))=\mathrm{span}\{ v^* \}. 
    \]
    \item[ii)] The following non-degeneracy assumption holds: 
    \begin{equation} \label{eq:graphonnondegenerate}
        \frac{\int_0^1\int_0^1 W(x,y)\sin(u^*(y)-u^*(x))\left(v^*(y)-v^*(x)\right)^2v^*(x) \mathrm{d}x\mathrm{d}y }{\int_0^1 (\Omega(x)-\overline\Omega)v^*(x)\mathrm{d} x}<0.
    \end{equation}
    \end{enumerate}  
\end{hyp}

The above assumption guarantees the existence of a saddle-node bifurcation in the graphon model at $K = \Kcrit$. As shown in Lemma~\ref{lem:CMgraphon} below, the sign condition in Hypothesis~\ref{hyp:SNgraphon}(ii) implies the existence of steady-state synchronized solutions for $K>\Kcrit$ to \eqref{FFunction}. Alternately, one could reverse the sign to have synchronous solutions for $K < \Kcrit$, but we have opted for the current presentation to best reflect our applications in the next section. Furthermore, we have assumed for simplicity that all other elements of the spectrum of $DF(u^*,\overline\Omega,\Kcrit)$ in $X$ beyond the bifurcation eigenvalue are contained in the left half of the complex plane. These results are easily extended to the case where there is a finite collection of eigenvalues in the right half of the complex plane, but the reason we have chosen to omit this case is simply for the ease of presentation. No further technical hurdles exist should there be eigenvalues in the right half of the complex plane. We now present the following result which uses the notation $B_\delta(v)$ to denote the ball of radius $\delta > 0$ about the vector $v$.

\begin{thm}\label{thm:Bifurcation} 
    Assume Hypotheses~\ref{hyp:ConvergenceSequences} and \ref{hyp:SNgraphon}. There exists a $\delta > 0$ such that for all $\rho > 0$ there exists an $N \geq 1$ so that for every $n \geq N$ the following is true. There is a $\Kncrit > 0$ and vector $({\bf u}_n^*,\omega^*_n) \in \mathbb{R}^n\times \mathbb{R}$ satisfying $G_n({\bf u}_n^*,\omega^*_n,\Kncrit) = 0$ and $\max\{\|u_n^* - u^*\|_\infty,|\overline\Omega - \omega^*_n|,|\Kncrit - \Kcrit|\} < \rho$, where $u_n^*\in X$ is the step function representation of ${\bf u}_n^*$ over $\{I_j^n\}_{j = 1}^n$. Furthermore, there exists two smooth and distinct branches of solutions to $G_n = 0$ emanating from $({\bf u}_n^*,\omega^*_n,\Kncrit)\in \mathbb{R}^n\times\mathbb{R}\times\mathbb{R}$ that exists for all $K \in [\Kncrit,\Kcrit + \delta]$, while there are no solutions of $G_n = 0$ in $B_\delta({\bf u}_n^*,\omega^*_n,\Kncrit)$ for all $K \in (\Kcrit - \delta,\Kncrit)$.          
\end{thm}

The proofs of the above theorems are left to the latter sections of this paper, before which we demonstrate an application of our results in the following section. We will first prove Theorem~\ref{thm:Bifurcation} in Section~\ref{sec:CentreManifoldProof}. Then, in Section~\ref{sec:PersistenceProof} we provide a commentary on the proof of Theorem~\ref{thm:Persistence} being similar, but ultimately simpler. That is, the proof of Theorem~\ref{thm:Bifurcation} is shown to be a more involved version of Theorem~\ref{thm:Persistence} since one must account for the zero eigenvalue assumed by Hypothesis~\ref{hyp:SNgraphon}. This will be explained in more detail as one proceeds through the proof sections.

\section{Applications to Erd\H{o}s--R\'enyi Networks}\label{sec:ApplicationER}

In this section we apply our results to Erd\H{o}s--R\'enyi networks. The goal here is to work with a simplified model that can elucidate much of our theory, while also providing some technical details that indicate that bifurcations to synchrony in the graphon model do not always come from a simple saddle-node bifurcation. That is, we see that different choices for the function $\Omega$ in \eqref{FFunction} can lead to bifurcations to synchrony coming from the essential spectrum, which in turn means that our results in Theorem~\ref{thm:Bifurcation} cannot be applied as the assumptions are not satisfied. Nonetheless, away from these bifurcations our results in Theorem~\ref{thm:Persistence} can always be applied to demonstrate persistence of synchronous states onto large finite networks of coupled oscillators.

\subsection{Synchronous States and the Critical Threshold}\label{sec:Ermentrout}

To begin, let us suppose that frequencies are drawn from a distribution with probability density function $f:[-1,1] \to \mathbb{R}$. Then, the cumulative distribution function is given by 
\begin{equation}
    F(x) = \int_{-1}^x f(s)\mathrm{d}s,
\end{equation}
so that the connection between the frequencies in \eqref{Kuramoto} and the function $\Omega$ in \eqref{KuramotoGraphon} then comes from setting $\Omega(x) = F^{-1}(x)$. 

Such a connection was already established in Ermentrout's pioneering work \cite{ermentrout1985synchronization}.  Depending on the choice of $\Omega$, our results herein return Ermentrout's result to the finite-dimensional system to show that the Kuramoto critical coupling on large all-to-all networks is well-estimated by that of the graphon models' for randomly distributed frequencies. This further complements the results of \cite{dorfler2011critical} which provides asymptotically non-sharp bounds in $n$ on the critical coupling for all-to-all networks by providing the precise limiting critical coupling value as $n \to \infty$ (with high probability).

These results go beyond \cite{ermentrout1985synchronization,dorfler2011critical} by being applicable to random networks as well. To illustrate, consider \eqref{KuramotoGraphon} with the Erd\H{o}s--R\'enyi graphon $W(x,y) = p$, for some $p \in (0,1]$, for all $(x,y) \in [0,1]^2$. The corresponding finite-dimensional Kuramoto model \eqref{Kuramoto} is posed on a randomly generated network $\mathbb{G}(n,p)$ which assigns edges $A_{i,j} = A_{j,i} = 1$ with probability $p$. The  continuum model takes the form 
\begin{equation}\label{KuramotoER}
    \frac{\partial \theta}{\partial t} = \Omega(x) + Kp\int_0^1 \sin(\theta(y,t) - \theta(x,t))\mathrm{d}y.       
\end{equation}
Note that with $p = 1$ we are in the all-to-all setting of \cite{ermentrout1985synchronization}, while $p < 1$ simply acts to scale the coupling coefficient in the  continuum model. Thus, we can provide the following adaptation of Ermentrout's analysis to characterize the conditions that guarantee whether or not synchronous patterns exist in \eqref{KuramotoGraphon} with an Erd\H{o}s--R\'enyi graphon.

\begin{prop}\cite[Proposition~2]{ermentrout1985synchronization}\label{prop:Ermentrout} 
    Let $\gamma = \sup |\Omega(x) - \overline{\Omega}|/(Kp)$, where $\overline{\Omega} = \int_0^1 \Omega(x)\mathrm{d}x$. A solution to \eqref{KuramotoER} is
    \begin{equation}
        \theta(x,t) = \overline{\Omega} t + u(x)
    \end{equation}
    where $\sin(u(x)) = [\Omega(x) - \overline{\Omega}]/(Kpq\gamma)$ and $(\gamma,q)$ are related through
    \begin{equation}\label{GammaEqn}
        \gamma = \frac{1}{q^2}\int_{-1}^1 \sqrt{q^2 - s^2}f(s)\mathrm{d}s,
    \end{equation}
    where $f = \frac{\mathrm{d}}{\mathrm{d} x} \Omega^{-1}$ is the probability density function of the frequencies. If $\gamma > 1$, synchronization will not occur for this solution. 
\end{prop}

Ermentrout goes on to study the synchronization threshold, i.e. the smallest value of $K$ for which these synchronous solutions exist, using the equation \eqref{GammaEqn}. Indeed, the onset of synchronization happens at
\begin{equation}\label{gammaStar}
    \gamma^* = \max_{q \geq 1} \frac{1}{q^2}\int_{-1}^1 \sqrt{q^2 - s^2}f(s)\mathrm{d}s, 
\end{equation}
which in turn leads to $\Kcrit = \frac{1}{p\gamma^*}$, where we note the inclusion of the graphon probability parameter $p \in (0,1]$. With this information, we provide the following lemma that can be used to verify the hypotheses of our main theorems.  We draw the attention of the reader to \cite{mirollo05,mirollo07} for analogous stability results and calculations of the phase locked state in discrete and  continuum versions of the Kuramoto model.   For simplicity, we will restrict ourselves to functions $\Omega$ that are odd over the midpoint $x = 1/2$, which is equivalent to considering probability distributions $f$ that are even over $[-1,1]$. This further gives that $\overline\Omega = 0$.

\begin{lem}\label{lem:StableER}
    Suppose that $\Omega(x)$ is odd over $x = 1/2$, $K\geq \Kcrit$ and let $u^*(x)$ be the synchronous solution of \eqref{FFunction} guaranteed by Proposition~\ref{prop:Ermentrout}. Then, the spectrum of  
    \begin{equation}
        DF(u^*,\overline\Omega,K)v = Kp\int_0^1\cos(u^*(y)-u^*(x))[v(y)-v(x)]\mathrm{d}y,   
    \end{equation} 
    as an operator on $X$ is real and broken into the essential and point spectrum (eigenvalues). Defining $\kappa = Kpq\gamma$ and $C = \int_0^1 \cos(u^*(y))\mathrm{d}y$, we have 
    \begin{enumerate}
        \item The essential spectrum is given by the interval $\sigma_{\mathrm{ess}} = \left[-KpC, -Kp\sqrt{1-\frac{1}{\kappa^2}}C \right]$
        \item $0$ is an eigenvalue of $DF(u^*,\overline\Omega,K)$ on the space $X$ if
        \begin{equation}\label{ZeroEigCondition}
            \frac{1}{\kappa C} \int_0^1 \frac{\Omega^2(y)}{\sqrt{\kappa^2-\Omega^2(y)}}\mathrm{d}y = 1.
        \end{equation}
    \end{enumerate}
Moreover, the synchronous solution is stable if 
    \[
        \frac{1}{\kappa C} \int_0^1 \frac{\Omega^2(y)}{\sqrt{\kappa^2-\Omega^2(y)}}\mathrm{d}y < 1.
    \]
\end{lem}

\begin{proof} 
The even symmetry of the cosine coupling function in the linearization $DF(u^*,\overline\Omega,K)$ allows one to conclude that the operator is self-adjoint on $L^2$. Thus, the spectrum is entirely real as an operator on $L^2$. Further, \cite[Lemma~5.1]{bramburger2024persistence} proves that the spectrum of this linearization is equivalent on $L^2$ and $C([0,1])$, and since $X$ is a subspace of $C([0,1])$, it follows that the spectrum of $DF(u^*,\overline\Omega,K):X \to X$ is contained in the real line. We now proceed to characterize parts of this real spectrum and prove the stated proposition.

First, recall from the discussion above that $\overline\Omega = \int_0^1\Omega(x)\mathrm{d}x=0$ since we are assuming that $\Omega(x)$ is odd over $x = 1/2$. Then, for $K \geq \Kcrit$ we recall from Proposition~\ref{prop:Ermentrout} that $u^*(x) = \arcsin \left(\Omega(x)/\kappa\right)$, where $\kappa = Kpq\gamma\geq 1$ is as given in the statement of the lemma. Using the angle difference identity for cosine we obtain
\begin{equation}
\begin{split}
    DF(u^*,0,K)v&=Kp\cos(u^*(x))\int_0^1 \cos(u^*(y))v(y)\mathrm{d}y +Kp\sin(u^*(x))\int_0^1 \sin(u^*(y))v(y)\mathrm{d}y \\
    &\quad - \left(Kp\cos(u^*(x)) \int_0^1 \cos(u^*(y))\mathrm{d}y\right)  v(x) 
\end{split}
\end{equation}
A spectral decomposition of operators of this form on the Banach Space $X$ was obtained in Lemma~4.1 of \cite{bramburger2024persistence}. In particular, the essential spectrum $\sigma_{ess}$, is comprised of the set of $\lambda \in \mathbb{C}$ lying in the range of the multiplication part of the operator. We therefore have 
\[ 
    \sigma_{\mathrm{ess}} = \left[-Kp \int_0^1 \cos(u^*(y))\mathrm{d}y, -Kp\sqrt{1-\frac{1}{\kappa^2}} \int_0^1\cos(u^*(y))\mathrm{d}y \right],
\]
which when $\kappa>1$ the above interval is a strict subset of $(-\infty,0)$, while for $\kappa=1$ it includes the point $0$.  

Next, we study the point spectrum (eigenvalues) of $DF(u^*,0,K)$. We therefore seek continuous functions $v^*(x)$ such that $DFv^*=\lambda v^*$.  To condense notation we let $c(x)=\cos(u^*(x))$ and note that $\sin(u^*(x))=\Omega(x)/\kappa$. It therefore holds that any eigenpair $(v^*,\lambda)$ must satisfy
\[ 
    \lambda v^*(x)=Kp c(x)\int_0^1 c(y)v^*(y)\mathrm{d}y +Kp\frac{1}{\kappa} \Omega(x)\int_0^1 \frac{1}{\kappa} \Omega(y) v^*(y) \mathrm{d}y -Kp c(x)v^*(x) \int_0^1c(y)\mathrm{d}y. 
\]
Letting $\lambda^*=\frac{\lambda}{Kp}$ then we reduce to finding functions $v^*$ satisfying
\begin{equation}\label{eq:ABCformforv} 
    \lambda^* v^*(x)=Ac(x)+B\Omega(x)-Cc(x) v^*(x), 
\end{equation}
where
\begin{equation}\label{ABCdefinitions}
    \begin{split}
        A&= \int_0^1 c(y)v^*(y)\mathrm{d}y \\
        B&= \frac{1}{\kappa^2}\int_0^1 \Omega(y) v^*(y)\mathrm{d}y \\
        C& = \int_0^1 c(y)\mathrm{d} y. 
    \end{split}
\end{equation}
From \eqref{eq:ABCformforv} we find that 
\[ 
    v^*(x)=\frac{Ac(x)+B\Omega(x)}{Cc(x)+\lambda^*}. 
\]
Now, since $\Omega(x)$ is odd over $x = 1/2$, it follows that $c(x)$ is even over $x = 1/2$, and so we have the useful fact
\[ 
    \int_0^1 \frac{c(y)\Omega(y)}{Cc(y)+\lambda^*}\mathrm{d}y =0
\]
With this fact we obtain the solvability conditions 
\begin{equation}\label{SolvabilityConditions}
    \begin{split}
        A&=A\int_0^1 \frac{c^2(y)}{Cc(y)+\lambda^*}\mathrm{d}y \\
        B&=\frac{B}{\kappa^2} \int_0^1 \frac{\Omega(y)^2}{Cc(y)+\lambda^*}\mathrm{d}y
    \end{split}
\end{equation}

Focusing on the first integral in \eqref{SolvabilityConditions} we note that when $\lambda^*=0$ the condition reduces to 
\[ 
    \frac{1}{C} \int_0^1 c(y)\mathrm{d}y=1,  
\]
which holds due to the definition of the constant $C$ in \eqref{ABCdefinitions}. We therefore recover that $v^*(x)=1$ is an eigenfunction of the operator $DF(u^*,0,K)$ with eigenvalue zero, corresponding to the translational invariance in the phase variable. However, this function does not lie in the space $X$ as it does not have mean zero. We therefore turn to the second integral in \eqref{SolvabilityConditions} and focus on solutions of $I(\lambda^*)=1$ where
\begin{equation}\label{IEquation}
    I(\lambda^*)=\frac{1}{\kappa^2}\int_0^1 \frac{\Omega^2(y)}{Cc(y)+\lambda^*}\mathrm{d}y
\end{equation}
From \eqref{SolvabilityConditions} the candidate function $v^*(x)$ is an eigenfunction if $I(\lambda^*)=1$. Taking $\lambda^* = 0$ then requires one to solve 
\[ 
     1 = \frac{1}{\kappa C} \int_0^1 \frac{\Omega^2(y)}{\sqrt{\kappa^2-\Omega^2(y)}}\mathrm{d}y, 
\]
as stated in the lemma. Finally, since eigenvalues occur whenever $I(\lambda^*)=0$, the fact that $\partial_{\lambda^*}I(\lambda^*)<0$ combined with $I(0)<1$ implies stability of the synchronous solution.  This completes the proof.
\end{proof} 

One can see from the above result that if $\kappa = 1$ then the essential spectrum of the linearization $DF(u^*,\overline\Omega,K)$ contains 0, thus meaning that neither of our results can be applied. As it turns out, not all bifurcations to synchrony in the graphon model \eqref{KuramotoER} are the result of a simple saddle-node bifurcation, as described in Theorem~\ref{thm:Bifurcation}, but can be attributed to bifurcations from the essential spectrum. In the following subsections we elucidate these cases in more detail for the reader, showing when our results can be applied and when they cannot.

\subsection{Bifurcations from the Essential Spectrum}\label{sec:EssentialSpectrum}


 In this subsection, we provide additional insight into the onset of synchronization predicted by Proposition~\ref{prop:Ermentrout} in regimes where our theoretical analysis does not apply. Specifically, we show that there exist frequency distributions for which the onset of synchrony in the continuum limit is not governed by a codimension-one saddle-node bifurcation, but instead by a more intricate mechanism involving the essential spectrum. Despite this increased analytical complexity, the numerical results presented here suggest that the central features of our main conclusions persist. In particular, the simulations indicate that both the critical coupling strength and the qualitative shape of the synchronized solutions are still well captured by the continuum limit, including for random graph realizations, even when synchronization arises through bifurcations involving the essential spectrum.

Let us begin by considering the case where frequencies are drawn from the uniform distribution on $[-1,1]$. In this case the probability density function is $f(\omega) = \frac{1}{2}$, resulting in $\Omega(x) = 2x -1$. One can further find that $\gamma^*$ in \eqref{gammaStar} occurs when $q = 1$, meaning that $\kappa = Kpq\gamma$ takes the value $\kappa = 1$ at the critical coupling $\Kcrit = \frac{1}{p\gamma^*} = \frac{4}{\pi p}$. Thus, we find that at this critical coupling parameter the essential spectrum of the linearization $DF(u^*,\overline\Omega,\Kcrit)$ does not satisfy the conditions to apply Theorem~\ref{thm:Bifurcation}.

We can take this further by performing general calculations for any $\kappa \geq 1$. With $\Omega(x) = 2x - 1$ we first have that 
\[ 
    C=\int_0^1\sqrt{1-\left(\frac{2x-1}{\kappa}\right)^2}\mathrm{d}x =\frac{1}{2\kappa}  \int_{-1}^1\sqrt{\kappa^2-y^2}\mathrm{dy}=\frac{\kappa}{4} \left[\arcsin\left(\frac{y}{\kappa}\right) +\frac{y}{\kappa^2}\sqrt{\kappa^2-y^2}  \right]_{-1}^1, 
\]
so that 
\[ 
    C=\frac{\kappa}{2} \arcsin\left({\frac{1}{\kappa}}\right) +\frac{1}{2\kappa} \sqrt{\kappa^2-1}. 
\]
On the other hand,  computing the integral in (\ref{IEquation}) we obtain  
\[ 
    \frac{1}{\kappa} \int_0^1\frac{(2x-1)^2}{\sqrt{\kappa^2-(2x-1)^2}}\mathrm{d}x =\frac{1}{\kappa}\int_{-1}^1 \frac{y^2}{\sqrt{\kappa^2-y^2}}\mathrm{d}y=\frac{1}{\kappa}\left[ -y \sqrt{\kappa^2-y^2}\right]_{-1}^1+\frac{1}{2\kappa}\int_{-1}^1 \sqrt{\kappa^2-y^2}\mathrm{d}y.   
\]
Notice that the final integral in the above expression is exactly $C$ from above, and so putting this all together we obtain
\begin{equation} \label{eq:Icompinbasiccase} 
    \frac{1}{C} \frac{1}{\kappa}\int_0^1 \frac{\Omega^2(y)}{\sqrt{\kappa^2-\Omega^2(y)}}\mathrm{d}y = 1-\frac{1}{C\kappa} \sqrt{\kappa^2-1}.
\end{equation}
Thus, from Lemma~\ref{lem:StableER} we see that there can only be a zero eigenvalue when $\kappa = 1$, which is the case discussed previously. The implication of these calculations is that for any $\kappa>1$ the synchronous solution is spectrally stable with respect to perturbations in the Banach space $X$ and further satisfies the conditions to apply Theorem~\ref{thm:Persistence}. 

\begin{figure}[t] 
\center
\includegraphics[width = 0.75\textwidth]{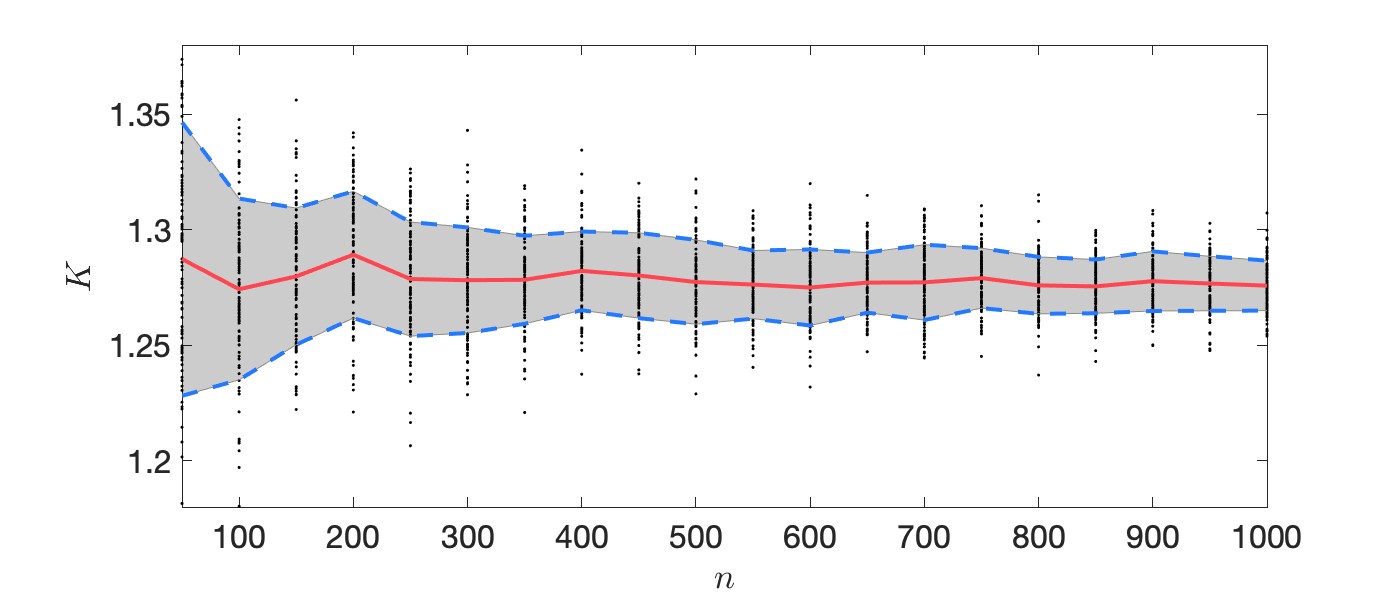} \\  
\includegraphics[width = 0.75\textwidth]{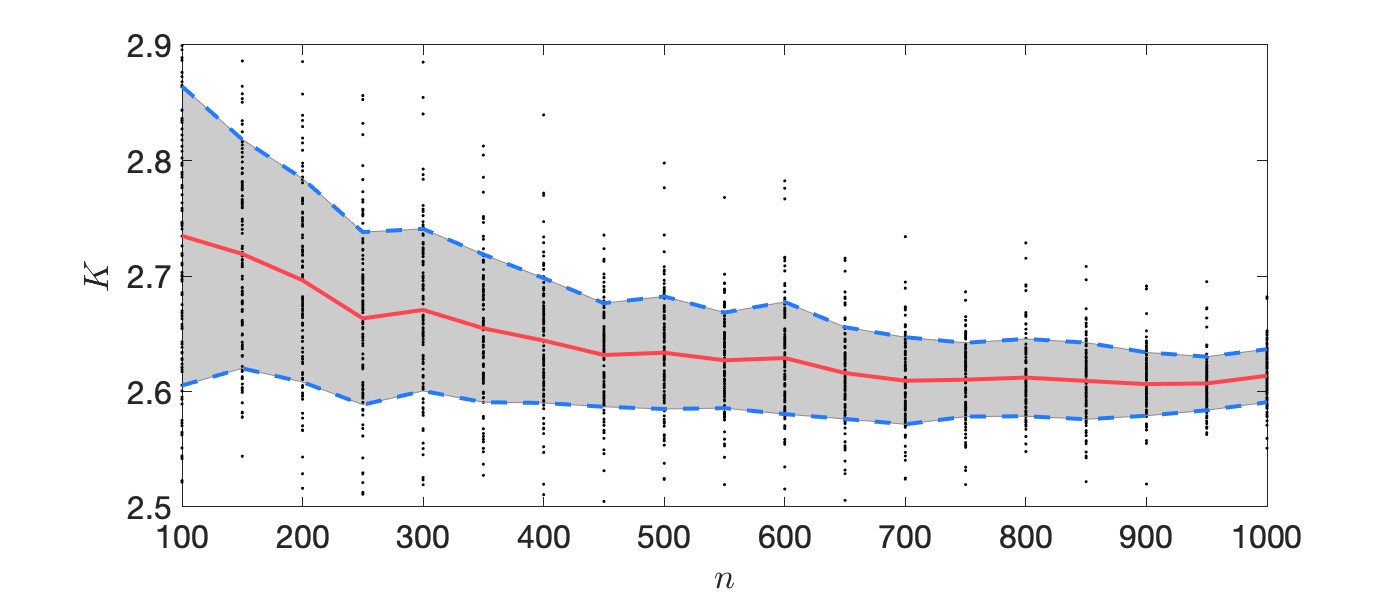}   
\caption{Each black dot on both figures is the location of the critical coupling value of \eqref{Kuramoto} on an Erd\H{o}s--R\'enyi graph with edge probability $p = 1$ (top) and $p = 0.5$ (bottom) and frequencies $\omega_j$ drawn from the uniform distribution on $[-1,1]$. For each network size $n$ there are 100 random realizations of system \eqref{Kuramoto}, with the red line representing the mean across $n$ and the shaded region bounded by blue dashed lines denoting one standard deviation from the mean.}
\label{fig:ER_CriticalTheshold}
\end{figure}

Our result in Theorem~\ref{thm:Persistence} can be applied to any $K > \Kcrit$ to provide stable synchronous solutions over random Kuramoto networks. However, since the essential spectrum of the linearization includes zero at $K = \Kcrit$, we cannot analytically confirm the proximity of the onset of synchrony in random networks to that of the graphon model. Nonetheless, we are able to provide numerical results that appear to confirm that our results of Theorem~\ref{thm:Bifurcation} still hold. Figure~\ref{fig:ER_CriticalTheshold} presents the identified critical coupling for 100 realizations of random Kuramoto models of size $n = 50,100,\dots,1000$, represented by black dots. We also provide the mean critical coupling value (red line) and the shaded region enclosed by blue lines represents one standard deviation from the mean. We provide results for all-to-all networks ($p = 1$) and Erd\H{o}s--R\'enyi random networks $(p = 0.5)$. For all-to-all networks the mean critical coupling at $n = 1000$ is 1.2758, compared with the graphon value of $4/\pi \approx 1.2372$, while the Erd\H{o}s--R\'enyi networks have mean 2.6137, compared with their graphon value of $8/\pi\approx 2.5465$. While the relative error for Erd\H{o}s--R\'enyi networks is only $\sim 2\%$, this larger value compared to the all-to-all networks is attributed to the slower convergence of Erd\H{o}s--R\'enyi graphs to their graphon in the cut norm.      

\begin{figure}[t] 
\center
\includegraphics[width = 0.45\textwidth]{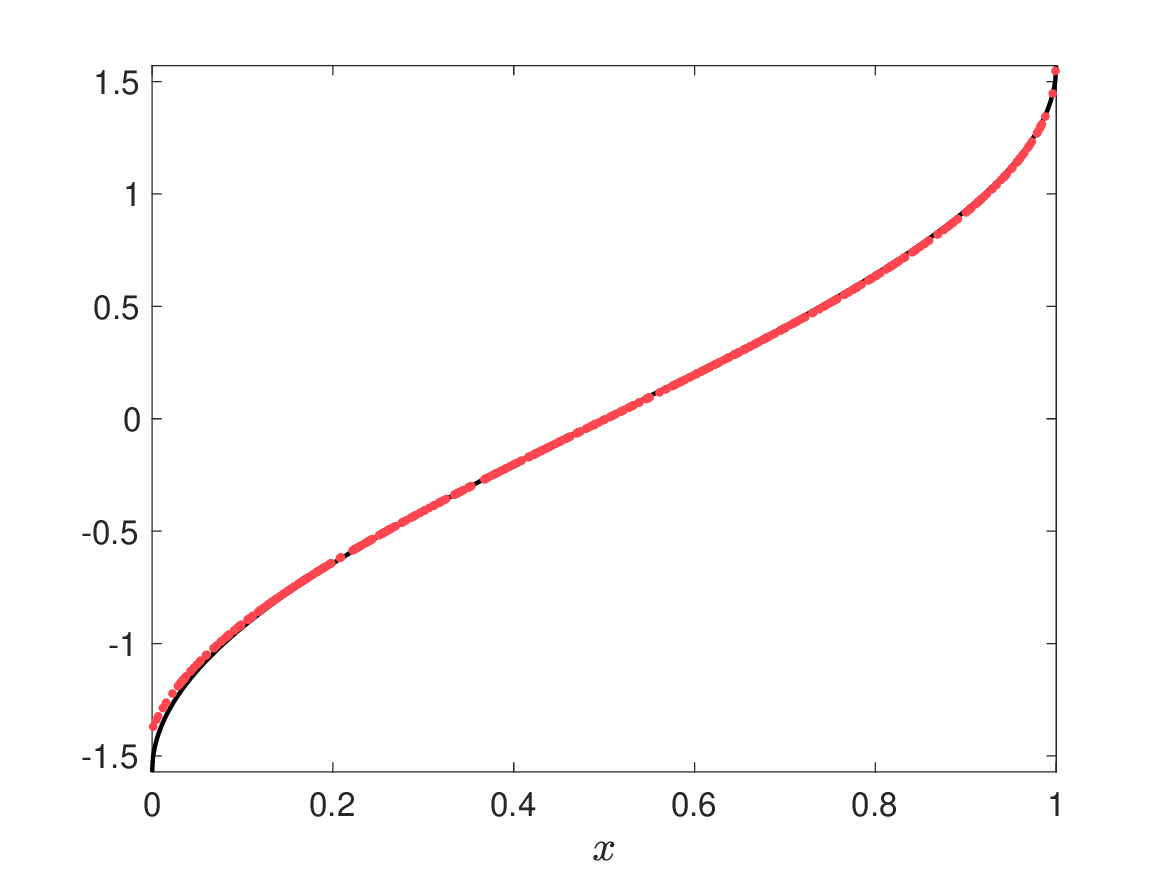} 
\includegraphics[width = 0.45\textwidth]{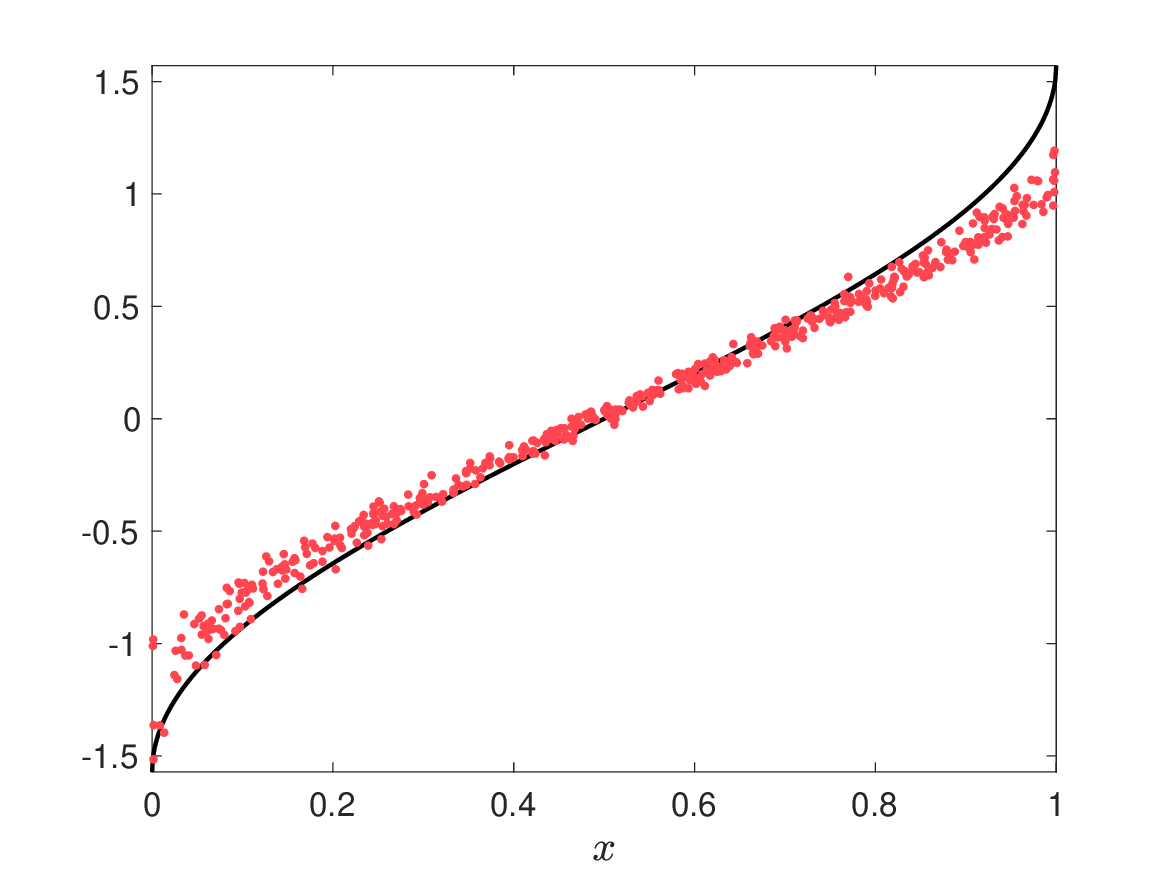}   
\caption{Comparison of the synchronous solution at the critical coupling with $n = 500$ oscillators and frequencies drawn from the uniform distribution (red dots) against the continuum synchronous profile $\theta(x) = \arcsin(2x - 1)$ (black line) to \eqref{KuramotoER}. Synchronous solutions are plotted as $\{(x_j,\theta_j)\}_{j = 1}^{500}$ with each $x_j$ drawn independently from the uniform distribution on $[0,1]$ to generate the frequencies $\omega_j = \Omega(x_j)$. Left: All-to-all coupling ($p=1$). Right: Erd\H{o}s--R\'enyi random network $(p =0.5)$.}
\label{fig:ER_Sols}
\end{figure}

We can also use Proposition~\ref{prop:Ermentrout} to compare the profiles of the synchronous solutions in random Kuramoto models to those of the graphon model \eqref{KuramotoER}. Figure~\ref{fig:ER_Sols} compares the graphon solution $u(x) = \arcsin(2x - 1)$ to the synchronous profiles on a $n = 500$ oscillator network at its critical coupling value. Synchronous solutions are plotted as $\{(x_j,\theta_j)\}_{j = 1}^{500}$, where $x_j$ are drawn independently from the uniform distribution on $[0,1]$. Recall that the frequencies are given by $\omega_j = \Omega(x_j)$. Again, we see strong agreement with the graphon prediction, particularly in the case of all-to-all networks ($p = 1$). The Erd\H{o}s--Renyi network $(p = 0.5)$ shows more random fluctuations in the profile, coming from the random network topology, but still retains the same basic profile as that predicted by the graphon model.

\begin{figure}[t] 
\center
\includegraphics[width = 0.45\textwidth]{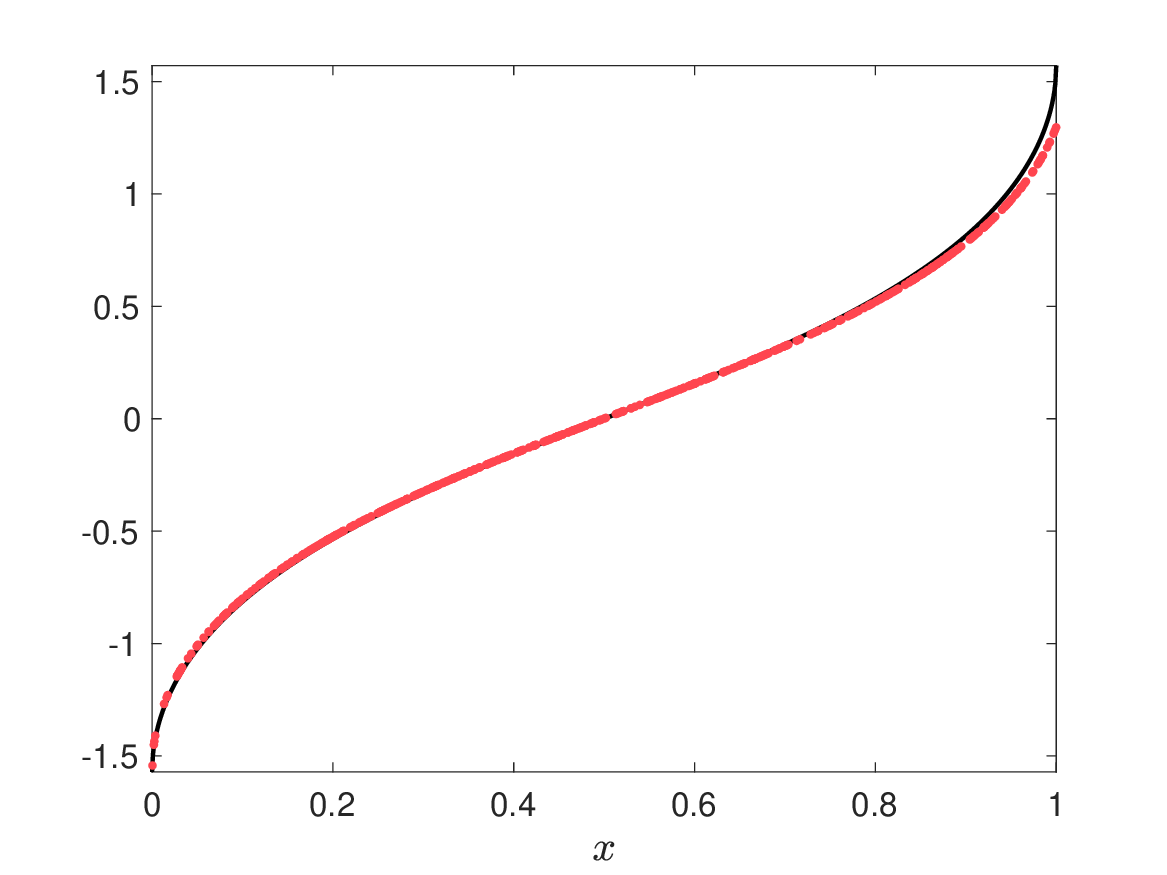} 
\includegraphics[width = 0.45\textwidth]{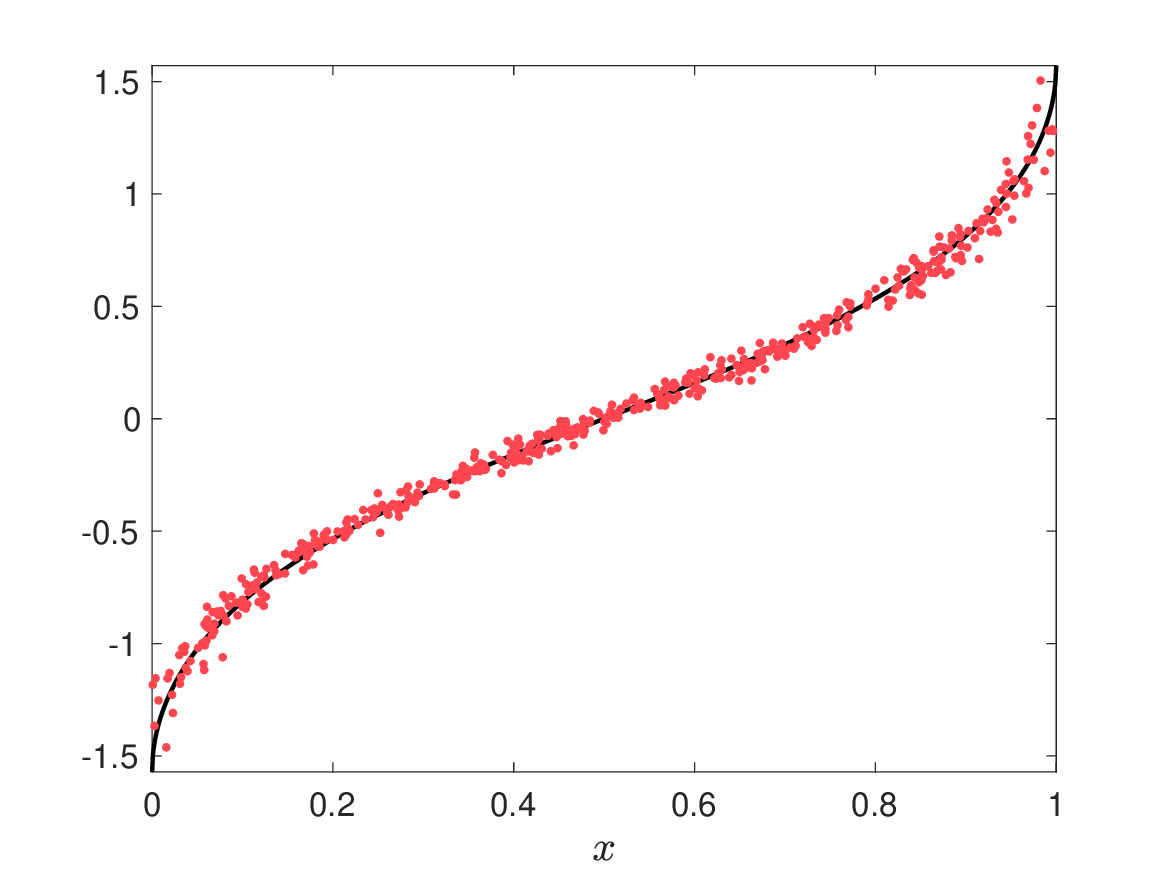}   
\caption{Comparison of the synchronous solution at the critical coupling with $n = 500$ oscillators and frequencies drawn from the Cauchy distribution with density $f(\omega) = \frac{2}{\pi(1 + \omega^2)}$ (red dots) against the continuum synchronous profile $\theta(x) = \arcsin(\tan(\frac{\pi}{4}(2x - 1)))$ (black line) to \eqref{KuramotoER}. Synchronous solutions are plotted as $\{(x_j,\theta_j)\}_{j = 1}^{500}$ with each $x_j$ drawn independently from the uniform distribution on $[0,1]$ to generate the frequencies $\omega_j = \Omega(x_j)$. Left: All-to-all coupling ($p=1$). Right: Erd\H{o}s--R\'enyi random network $(p =0.5)$.}
\label{fig:Cauchy_Sols}
\end{figure}

Having $\kappa =1$ at the critical coupling value is not unique to uniformly distributed frequencies. In \cite{ermentrout1985synchronization} Ermentrout identifies numerous distributions for which $\gamma^*$ occurs when $q = 1$, thus having the essential spectrum of the linearization about the solution guaranteed by Proposition~\ref{prop:Ermentrout} touch the imaginary axis in the complex plane. Again we emphasize that our results in Theorem~\ref{thm:Persistence} can be applied away from $K = \Kcrit$, while despite Theorem~\ref{thm:Bifurcation} not being applicable to describe the onset of synchrony in random Kuramoto networks, it appears that similar results to Theorem~\ref{thm:Bifurcation} still hold in this more complex situation of bifurcations from the essential spectrum. 

To better emphasize the point here, we provide another demonstration. Consider frequencies drawn from the Cauchy distribution, $f(\omega) = \frac{2}{\pi(1 + \omega^2)}$, which in this case gives $\gamma^* \approx 0.8284$, occurring at $q = 1$ in \eqref{GammaEqn}. With this distribution \eqref{KuramotoER} has $\Omega(x) = \tan(\frac{\pi}{4}(2x - 1))$, giving a synchronous solution profile of $u(x) = \arcsin(\tan(\frac{\pi}{4}(2x - 1)))$ at the critical coupling value $\Kcrit \approx 1/(0.8284p)$. Figure~\ref{fig:Cauchy_Sols} presents the same results as Figure~\ref{fig:ER_Sols}, but now with frequencies drawn from the Cauchy distribution. The all-to-all $(p = 1)$ network of $n = 500$ oscillators has a critical coupling value of $1.2100$, compared with $\Kcrit$ taking the value $1.2071$, while this realization of an Erd\H{o}s--R\'enyi network $(p = 0.5)$ of $n = 500$ oscillators has critical coupling $2.4265$, with  $\Kcrit$ being $2.4272$ (since $p = 0.5$ here).

\subsection{Co-dimension One Bifurcations to Synchrony}


We now turn our attention to the situation where the bifurcation to synchrony occurs through a saddle-node bifurcation due to an isolated eigenvalues crossing the imaginary axis.  In particular, any probability distribution for which $\gamma^*$ in \eqref{gammaStar} occurs at a value $q > 1$ will necessarily give $\kappa > 1$, which in turn provides that the essential spectrum is bounded away from the imaginary axis per Lemma~\ref{lem:StableER}. This then allows for the application of our results in Theorem~\ref{thm:Bifurcation}, while away from any bifurcation point we further have the persistence results of Theorem~\ref{thm:Persistence}.

\begin{figure} 
\center
\includegraphics[width = \textwidth]{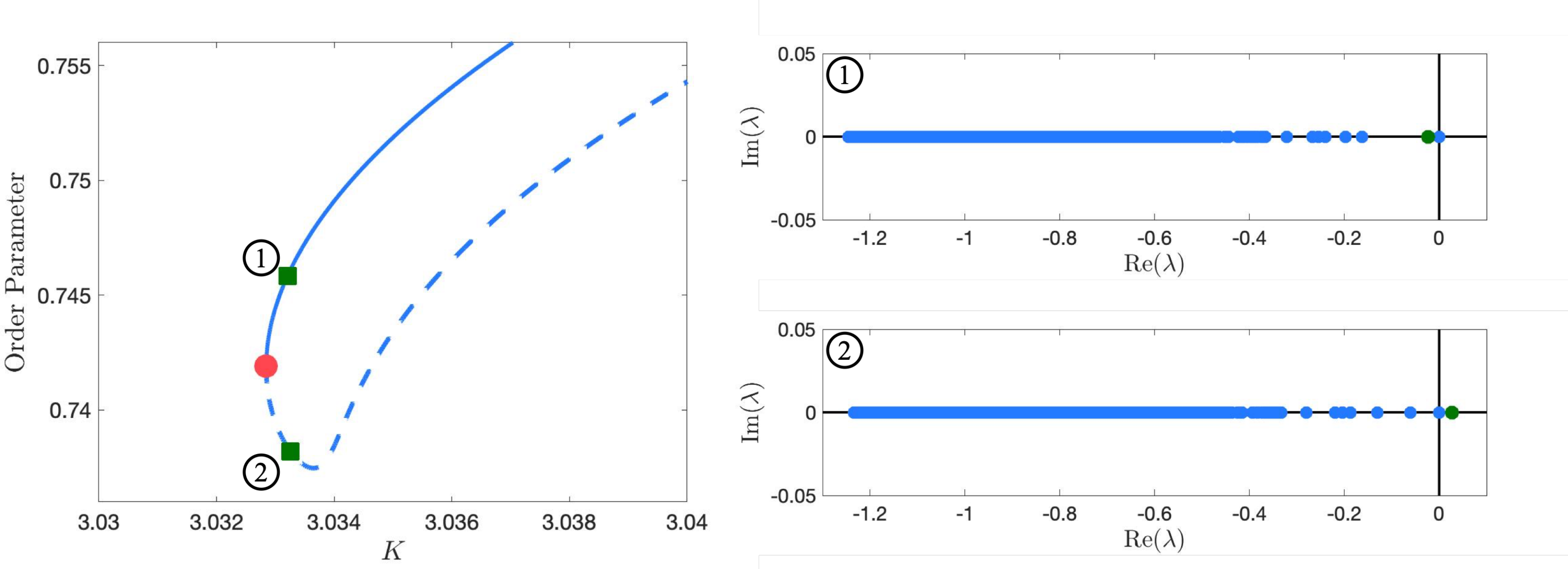}   
\caption{Left: Continuation of synchronous solutions in $n = 500$ oscillator random Kuramoto model on an Erd\H{o}s--R\'enyi network with $p = 0.5$ and frequencies distributed according to the density \eqref{CosineDistribution}. Plotted is the order parameter \eqref{OrderParam} versus the coupling coefficient $K$ with a saddle-node bifurcation leading to onset of synchronization denoted by a red dot at $K \approx 3.0328$. Linearized stability is indicated by a solid curve, while unstable solutions are along dashed curves. Right: Eigenvalues of the linearization about the synchronous solution at two points along the bifurcation curve indicated by green squares, showing a single eigenvalue cross at the bifurcation point (emphasized in green).}
\label{fig:ER_Cosine}
\end{figure}

For example, consider frequencies drawn from the distribution with density
\begin{equation}\label{CosineDistribution}
    f(\omega) = \frac{1}{\pi\sqrt{1 - \omega^2}},
\end{equation}
which has cumulative distribution function $F(\omega) = \frac{1}{2} - \frac{1}{\pi}\arcsin(\omega)$ and in turn gives $\Omega(x) = -\cos(\pi x)$. In this case one may compute that $\gamma^* \approx 0.6715$, occurring at $q \approx 1.1002$. Since the maximizing value of $q$ is larger than 1, it follows that at $\Kcrit \approx 1.4892/p$ we have $\kappa < 1$. Moreover, we can check that the zero eigenvalue condition \eqref{ZeroEigCondition} is true, thus giving a standard saddle-node bifurcation at $\Kcrit$ to synchronous solutions in the graphon model. At $K = \Kcrit$ we have $\kappa = q \approx 1.1002$, which gives
\[
    \begin{split}
        C = \int_0^1\sqrt{1 - \cos^2(\pi y)/\kappa^2} \mathrm{d}y &\approx 0.7388 \\
        \implies \ \frac{1}{\kappa C} \int_0^1 \frac{\cos^2(\pi y)}{\sqrt{\kappa^2-\cos^2(\pi y)}}\mathrm{d}y &\approx 1.0000, 
    \end{split}
\]
thus at least numerically confirming the presence of a zero eigenvalue at $K = \Kcrit$ which is separated from the essential spectrum.

In Figure~\ref{fig:ER_Cosine} we see our results in application through the continuation of synchronous solutions in an $n = 500$ oscillator Kuramoto model on an Erd\H{o}s--R\'enyi network with $p = 0.5$ and frequencies drawn from the distribution with density given by \eqref{CosineDistribution}. We plot the order parameter, given by 
\begin{equation}\label{OrderParam}
    r = \frac{1}{n}\bigg|\sum_{j = 1}^n \mathrm{e}^{\mathrm{i}\theta_j}\bigg|,
\end{equation}
at a synchronous solution against the coupling coefficient $K$. The bifurcation to synchrony in Figure~\ref{fig:ER_Cosine} comes from a saddle-node bifurcation at $K \approx 3.0328$, a 2\% relative error from the graphon prediction of $\Kcrit \approx 2.9784$. By following the eigenvalues of the Kuramoto system linearized about the synchronous solution, we find that a single eigenvalue crosses zero at the critical coupling point (denoted by a red dot)\footnote{There is always an eigenvalue at 0 that corresponds to the translational invariance of the Kuramoto system. Our analysis is performed in the space $X$ precisely to quotient out this eigenvalue.}. The synchronous solutions along the upper curve in the bifurcation diagram are the finite-dimensional analogues of the graphon solution $u^*$ given in Proposition~\ref{prop:Ermentrout}, as guaranteed to exist by Theorem~\ref{thm:Persistence}. The synchronous solution at the saddle-node bifurcation point is plotted in Figure~\ref{fig:Cosine_Sols} and compared to $\theta(x) = \arcsin(-\cos(\pi x)/1.1002)$, coming from Proposition~\ref{prop:Ermentrout} at $q = 1.1002$. For further comparison, we also plot a solution from an all-to-all network ($p = 1$) at its critical coupling point $K \approx 1.4756$, with a $1\%$ relative error of the graphon model prediction of $\Kcrit \approx 1.4892$. We do not include a bifurcation diagram for this case as it is nearly identical to that in Figure~\ref{fig:ER_Cosine}.  

\begin{figure} 
\center
\includegraphics[width = 0.45\textwidth]{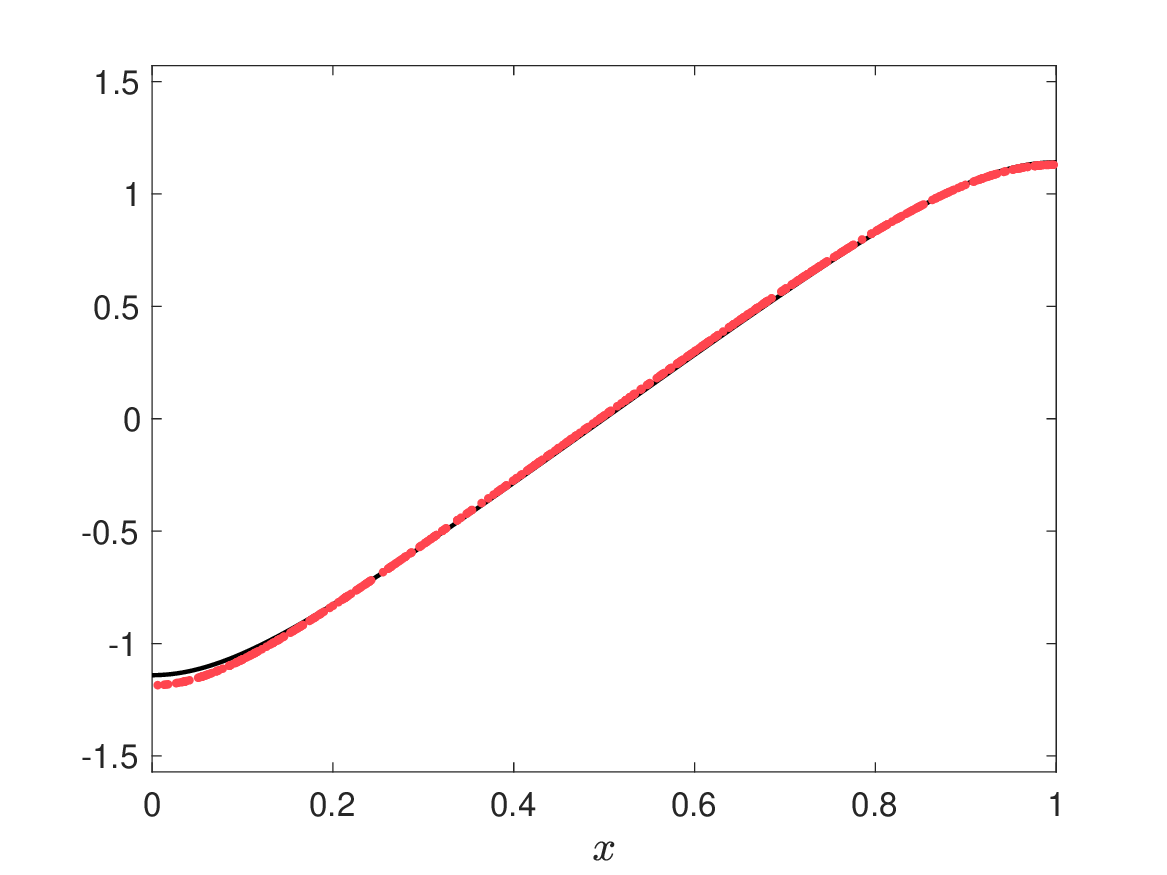} 
\includegraphics[width = 0.45\textwidth]{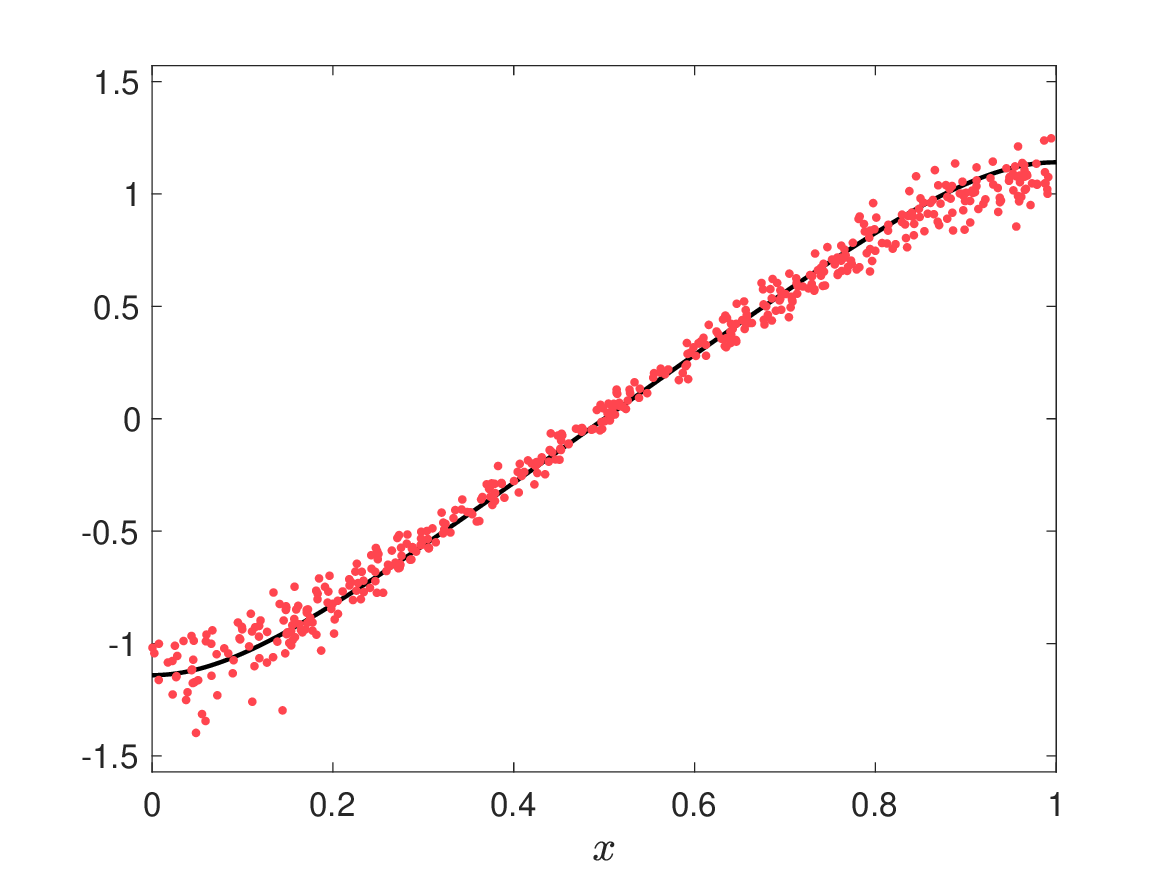}   
\caption{Comparison of the synchronous solution at the critical coupling with $n = 500$ oscillators and frequencies drawn from the distribution with density \eqref{CosineDistribution} (red dots) against the continuum synchronous profile $\theta(x) = \arcsin(-\cos(\pi x)/1.1002)$ (black line) to \eqref{KuramotoER}. Synchronous solutions are plotted as $\{(x_j,\theta_j)\}_{j = 1}^{500}$ with each $x_j$ drawn independently from the uniform distribution on $[0,1]$ to generate the frequencies $\omega_j = \Omega(x_j)$. Left: All-to-all coupling ($p=1$). Right: Erd\H{o}s--R\'enyi random network $(p =0.5)$.}
\label{fig:Cosine_Sols}
\end{figure}

\section{Proof of Theorem~\ref{thm:Bifurcation}}\label{sec:CentreManifoldProof}

In this section, we prove Theorem~\ref{thm:Bifurcation}. In particular, we show that under Hypotheses~\ref{hyp:ConvergenceSequences} and \ref{hyp:SNgraphon} a saddle-node bifurcation to synchrony in the graphon model \eqref{KuramotoGraphon} implies the existence of a saddle-node bifurcation to synchrony in the discrete model \eqref{Kuramoto} for $n \gg 1$ occurring at a critical coupling constant $\Kncrit\to \Kcrit$ as $n \to \infty$. We consider the graphon equation 
\begin{equation} \label{eq:maingraphon}
    \frac{du}{dt}= \Omega(x) - \overline\Omega +K\int_0^1 W(x,y)\sin(u(y,t))-u(x,t)))\mathrm{d} y, 
\end{equation} 
with $\overline\Omega = \int_0^1 \Omega(x)\mathrm{d}x$ fixed throughout. Then, as in Section~\ref{sec:Results}, we define 
\[ 
    F(u,K)= \Omega(x) - \overline\Omega + K \int_0^1 W(x,y)\sin(u(y)-u(x))\mathrm{d} y,
\]
where we suppress the dependence of $F$ on $\overline\Omega$ since it is fixed throughout. By definition, steady-state solutions of \eqref{eq:maingraphon} at a fixed value of $K$ correspond to solutions of the equation $F(u,K) = 0$.  

To analyze the discrete problem we re-cast the adjacency matrix $A\in\mathbb{R}^{n\times n}$ as a step-graphon $W_n:[0,1]^2\to [0,1]$ and study the nonlocal equation 
\begin{equation} \label{eq:mainstepgraphon}
    \frac{du_n}{dt}=\Omega_n(x)-\omega^*_n +K\int_0^1 W_n(x,y)\sin(u_n(y,t)-u_n(x,t))\mathrm{d} y, 
\end{equation} 
where $\omega^*_n=\int\Omega_n(x)\mathrm{d}x$ is also fixed throughout and $W_n(x,y)$ is the step graphon representation of the graph. Notice that with this choice of $\omega^*_n$ we have 
\[
    |\overline{\Omega} - \omega_n^*| = \bigg|\int_0^1 [\Omega(x) - \Omega_n(x)]\mathrm{d}x \bigg| \to 0,
\]
as $n \to \infty$, coming from the assumption that $\|\Omega - \Omega_n\|_\infty \to 0$ as $n\to \infty$ in Hypothesis~\ref{hyp:ConvergenceSequences}. 

This section is broken down as follows. First, in \S~\ref{sec:CMgraphon} we provide a center manifold reduction for \eqref{eq:maingraphon} and prove that a saddle-node bifurcation takes place at $\Kcrit$ when we assume Hypothesis~\ref{hyp:SNgraphon}. Then, in \S~\ref{sec:CMStepgraphon} we perform a similar center manifold reduction for the step graphon model \eqref{eq:mainstepgraphon}, in turn showing that a saddle-node bifurcation takes place at $\Kncrit$ nearby $\Kcrit$ when $n \gg 1$. Finally, in \S~\ref{sec:PiecewiseConstant} we show that our obtained solutions to \eqref{eq:mainstepgraphon} are piecewise constant on the intervals $\{I_j^n\}_{j = 1}^n$, which in turn yields that they correspond to steady-states of the finite-dimensional equation \eqref{GFunction}, completing the proof.

\subsection{Center manifold Reduction for the Graphon Equation}\label{sec:CMgraphon}

Hypothesis~\ref{hyp:SNgraphon} outlines sufficient conditions so that the graphon equation (\ref{eq:maingraphon}) undergoes a saddle-node bifurcation describing the emergence of a stable synchronized state when the coupling constant $K$ exceeds the critical coupling constant $\Kcrit > 0$. Associated to this bifurcation is a local center manifold on which the reduced dynamics can be written in the canonical/normal form of a saddle-node bifurcation. Our main result will show that, for $n \gg 1$, the random graph system will also undergo a saddle-node bifurcation to synchrony at some critical value $\Kncrit$ near $\Kcrit$. The construction of a center manifold in that case will rely on properties of the center manifold for the graphon equation. Therefore, in this subsection we will review the construction of a center manifold for the graphon equation \eqref{eq:maingraphon} and the associated reduced dynamics on that center manifold.  

Recall the definition of the closed subspace $X$ of mean-zero functions in $C([0,1])$, introduced earlier as 
\[ 
    X=\left\{ u\in C([0,1]) \ | \ \int_0^1 u(x)\mathrm{d}x =0 \right\}. 
\]
Note that $F$ maps the Banach Space $X$ back into itself. Now, by assumption we have there is some $u^* \in X$ so that at $K = \Kcrit$ that $F(u^*,\Kcrit)=0$. Let $L$ denote the linearization of $F$ about $u^*$ at $(K,\overline\Omega) = (\Kcrit,\int_0^1\Omega(x)\mathrm{d}x)$, acting on functions $v \in X$ by 
\begin{equation} 
    Lv= \Kcrit \int_0^1 W(x,y)\cos(u^*(y)-u^*(x))(v(y)-v(x))\mathrm{d} y.\label{eq:Ldef} 
\end{equation}
We now state the center manifold result for the graphon equation.  

\begin{lem}\label{lem:CMgraphon} 
Under the assumptions of Hypothesis~\ref{hyp:SNgraphon}, there exists $\delta > 0$ and a decomposition $X=X^c\oplus X^s$ so that the system \eqref{eq:maingraphon} has a local center manifold described by the graph $\Psi:X^c\times [\Kcrit-\delta,\Kcrit+\delta]\to X^s$.  The graph is $C^k$ for any $k> 2$. Letting $K=\Kcrit+\tilde{K}$ with $\tilde{K} \in [-\delta,\delta]$, the reduced dynamics on this center manifold are described by the scalar ordinary differential equation
\[ 
    \frac{dw_c}{dt} = a \tilde{K} +b w_c^2 +\mathcal{O}(w_c^3,\tilde{K}w_c,\tilde{K}^2),
\]
where 
\begin{align}  
    a & = -\frac{1}{\Kcrit}\int_0^1 [\Omega(x)-\overline \Omega]v^*(x)\mathrm{d} x \nonumber= \int_0^1\int_0^1 W(x,y)\sin\left(u^*(y)-u^*(x)\right)v^*(x)\mathrm{d}y\mathrm{d}x \\
    b &= -\frac{\Kcrit}{2} \int_0^1\int_0^1 W(x,y)\sin(u^*(y)-u^*(x))\left(v^*(y)-v^*(x)\right)^2 v^*(x)\mathrm{d}y\mathrm{d}x. \label{eq:aandb}
\end{align}
Finally, $\mathrm{sign}\left(ab\right)<0$.
\end{lem}
\begin{proof}
    The result follows from an application of \cite[Theorem~3.3]{haragus11}. We verify that our system satisfies the conditions required for this result in Appendix~\ref{sec:graphonproof}.  We emphasize that the center manifold is local and only valid in some neighborhood of the origin. This requires the use of a cut-off function applied to the nonlinearity to control the Lipschitz constant of the nonlinearity; see Appendix B of \cite{haragus11} and \cite{vanderbauwhede92}.
    
    The sign condition on the coefficients $a$ and $b$ ensures that bifurcating steady-states exist for $K>\Kcrit$ as assumed in Hypothesis~\ref{hyp:SNgraphon}, while a reversal of the sign would simply give that the steady-states exist for $K < \Kcrit$ but not effect any other portion of the proof.  
\end{proof} 

\subsection{Center-Manifold Reduction for the Step Case}\label{sec:CMStepgraphon}

We now turn our attention to the step function Kuramoto system (\ref{eq:mainstepgraphon}). We will derive center manifold results in analogy to the one obtained for the graphon equation \eqref{eq:maingraphon} that hold for all $n \gg 1$. That is, in this section we will prove the existence of a center manifold and perform a reduction to it for the step function model \eqref{eq:mainstepgraphon}, eventually proving the existence of a saddle-node bifurcation for \eqref{eq:mainstepgraphon} occurring in a neighborhood of the graphon bifurcation point $(\Kcrit,u^*(x))$.  

To begin, let 
\[ 
    u_n(t,x)=u^*(x)+v_n(t,x), \quad K=\Kcrit+\hat{K}.
\]
Then the perturbation $v_n$ satisfies the equation
\begin{equation} 
    \frac{dv_n}{dt} = F_n(u_n,K) := \Omega_n(x) - \omega_n^* +(\Kcrit+\hat{K})\int_0^1 W_n(x,y)\sin(u^*(y) + v_n(t,y) - u^*(x) - v_n(t,x))\mathrm{d}y, \label{eq:vnmain} \end{equation}
where we again suppress the dependence on the frequency $\omega_n^*$ since it will be fixed for each $n \geq 1$ throughout. We perform our analysis in the Banach Space $X_n$ given by
\begin{equation}
    X_n=\left\{u\in L^\infty \ \bigg| \ \int_0^1 u(x)\mathrm{d}x=0 \ \text{and} \ u(x) \ \text{is continuous on each interval} \ \left[\frac{i-1}{n},\frac{i}{n}\right)  \right\}.
\end{equation}
The linearization of $F_n$ about $(u^*,\Kcrit)$ is thus denoted as $L_n$ and acts on functions $v \in X_n$ by
\[ 
    L_nv := DF_n(u^*,\Kcrit)v=\Kcrit \int_0^1 W_n(x,y)\cos(u^*(y)-u^*(x))(v(y)-v(x))\mathrm{d}y, 
\]
The following lemma characterizes the spectrum of the linear operator $L_n$. Recall that $v^* \in X $ is the kernel element of $DF(u^*,\Kcrit)$.  

\begin{lem} \label{lem:specDFn} 
There exists constants $\zeta>0$, $r\in(0,\zeta)$, and $\varepsilon_0 \in (0,\zeta - r)$ such that for all $\varepsilon \in (0,\varepsilon_0)$ there exists an $N \geq 1$ such that for all $n \geq N$ the following is true:
\begin{enumerate}
    \item The linear operator $L_n:X_n\to X_n$ has a simple eigenvalue $\lambda_n$ with $|\lambda_n| < \varepsilon$,
    \item The associated eigenfunction $v_n^*(x) \in X_n$, normalized such that $\int_0^1 [v_n^*(x)]^2\mathrm{d}x =1$, satisfies 
        \begin{equation} \label{eq:vnvnstarestimate}
            \|v_n^*-v^*\|_\infty < \varepsilon,  
        \end{equation}
    for all $n\geq N$,
    \item The remainder of the spectrum lies in the ball $\{z\in\mathbb{C} \ | \ |z+\zeta|<r\}$, and
    \item The spectral projection onto $v_n^*$ is 
        \[ 
            P^c_n f=v^*_n(x)\int_0^1 f(y) v_n^*(y)\mathrm{d}y=v^*_n(x)\langle f, v^*_n\rangle, 
        \]
    with the stable projection defined via $P^s_n=I-P^c_n$.
\end{enumerate}
\end{lem}

\begin{proof}
    The proof mimics that of Lemma~5.1 and 5.2 of \cite{bramburger2024persistence} and so we sketch the argument here. First, from \cite[Lemma~5.1]{bramburger2024persistence} the spectrum of the linear operator $L$ is equivalent on $C([0,1])$ and $L^2([0,1])$. Next, the arguments in \cite[Lemma~5.2]{bramburger2024persistence} give that $\| L-L_n\|_{2\to 2} \to 0 $ as $n\to \infty$. This follows from the assumption that $\|W_n-W\|_\square \to 0$, $\|d_{W_n}-d_W\|_\infty\to 0$, and $\|\Omega - \Omega_n\|_\infty \to 0$ as $n \to \infty$ in Hypothesis~\ref{hyp:ConvergenceSequences}. Recall that Hypothesis~\ref{hyp:SNgraphon} gives that $L$ has a single eigenvalue at $0$ with eigenfunction $v^*$ with the remainder of the spectrum bounded away from the imaginary axis. Putting all of this together gives that there exists constants $\zeta>0$, $r\in(0,\zeta)$, and $\varepsilon_0 > 0$ so that $B_{\varepsilon_0}(0)\cap \{z\in\mathbb{C} \ | \ |z+\zeta|<r\}$ is empty and, due to \cite[Theorem IV.3.1]{kato}, for any $\varepsilon\in (0,\varepsilon_0)$ there exists a $N$ sufficiently large so that for any $n\geq N$ the spectrum of $L_n: X_n \to X_n$ is contained in the set $\{z\in\mathbb{C} \ | \ |z+\zeta|<r\} \cup B_\varepsilon(0)$. The fact that $B_\varepsilon(0)$ contains a single isolated eigenvalue of $L_n:X_n \to X_n$ with algebraic multiplicity one follows from \cite[Theorem IV.3.16]{kato}. This result also implies that $\|v_n^*-v^*\|_2\to 0$. 
    
    To obtain \eqref{eq:vnvnstarestimate} we decompose the linear operators into the sum of a multiplication operator and an integral operator  as
    \begin{equation} \label{eq:LandLndecomp}
    \begin{split}
        Lv&= Q(x)v+T[v] \\
        L_nv&= Q_n(x)v+T_n[v].
    \end{split}
    \end{equation}
    Note that the essential spectrum of $L$ is exactly the range of $Q(x)$; see \cite[Lemma~4.1]{bramburger2024persistence}. Thus, $Q(x)\neq 0$ since the essential spectrum is assumed to be stable and belongs to the ball $\{z\in\mathbb{C} \ | \ |z+\zeta|<r\}$. Then,
    \begin{equation}
        \begin{split}
            Lv^*-L_nv^*_n&= Q(x)v^*(x)+\int_0^1 W(x,y)\cos(u^*(y)-u^*(x))v^*(y)\mathrm{d}y -Q_n(x)v_n^*(x)\\
            &\qquad -\int_0^1 W_n(x,y)\cos(u^*(y)-u^*(x))v_n^*(y)\mathrm{d}y \\
            &= Q(x)(v^*(x)-v^*_n(x))+(Q(x)-Q_n(x))v^*_n(x) \\
            &\qquad +\int_0^1 (W(x,y)-W_n(x,y))\cos(u^*(y)-u^*(x))v^*(y)\mathrm{d}y \\
            &\qquad + \int_0^1 W_n(x,y)\cos(u^*(y)-u^*(x))(v^*(y)-v^*_n(y))\mathrm{d}y.
        \end{split}
    \end{equation}
Since $Lv^*-L_nv^*_n=\lambda_nv^*_n$ we rearrange the above to obtain
\begin{equation}
\begin{split} Q(x)(v^*(x)-v^*_n(x))&= \lambda_nv^*_n +(Q_n(x)-Q(x))v^*_n(x) \\
            &-\int_0^1 (W(x,y)-W_n(x,y))\cos(u^*(y)-u^*(x))v^*(y)\mathrm{d}y \\
            &- \int_0^1 W_n(x,y)\cos(u^*(y)-u^*(x))(v^*(y)-v^*_n(y))\mathrm{d}y. \end{split}
            \end{equation}
Then $\|v^*-v^*_n\|_\infty$ may now be controlled by the supremum norms of the terms on the right hand side of the previous equation.  In particular, $|\lambda_n|\to 0$ as just shown and  $\|Q_n(x)-Q(x)\|_\infty\to 0$ due to \cite[Lemma 4.10]{bramburger2024persistence}. Since $v^*\in C([0,1])$ we also have that 
\[ 
    \left\|\int_0^1 (W(x,y)-W_n(x,y))\cos(u^*(y)-u^*(x))v^*(y)\mathrm{d}y\right\|_\infty  
\]
can be made arbitrarily small by taking $n$ sufficiently large; see  \cite[Lemma 4.7]{bramburger2024persistence}. Also, H\"older's inequality implies 
\[ 
    \lim_{n \to \infty}\left\| \int_0^1 W_n(x,y)\cos(u^*(y)-u^*(x))(v^*(y)-v^*_n(y))\mathrm{d}y\right\|_\infty \leq \lim_{n \to \infty}\|v^*-v_n^*\|_2 = 0. 
\]
We apply a similar argument  to $L_nv^*_n=\lambda_n v^*_n$, re-writing to obtain
\[ v^*_n(x)= - \frac{1}{Q_n(x)-\lambda_n} \int_0^1 W_n(x,y)\cos(u^*(y)-u^*(x))v^*_n(y)\mathrm{d}y.\]
Then after noting that $Q_n(x)-\lambda_n\neq 0$  applying H\"older's inequality to the second term proves that $\|v^*_n\|_2=1$ implies uniform boundedness of $\|v_n^*\|_\infty$ for all large $n$, which in turn can be used to show that both $\|\lambda_n v_n^*\|_\infty$ and $\|(Q_n-Q)v^*_n\|_\infty$ converge to $0$ as $n \to \infty$. This completes the proof.

\end{proof}

\begin{rmk}
    For an $N \geq 1$ sufficiently large, Lemma~\ref{lem:specDFn} provides the following decomposition of the space $X_n$ for all $n\geq N$: 
    \[ 
        X_n=X_n^c\oplus X_n^s, \qquad X^s_n=\mathrm{Rng}(P^s_n) = \mathrm{ker}(P^c_n)\subset X_n.
    \]
    Moving forward we will let $L^s_n=L_n|_{X^s_n}$, the restriction of $L$ to the stable space $X^s_n$.
\end{rmk}

We now proceed to transform system \eqref{eq:vnmain} into a form suitable for an application of the parameter-dependent center manifold theorem.  The first step is to introduce new coordinates so that the transformed system 
\begin{equation} \frac{d\tilde{v}_n}{dt}=H_n(\tilde{v}_n,\tilde{K}), \label{eq:Heqn} \end{equation}
satisfies $H_n(0,0)=0$.  To accomplish this, we will find solutions of the equations
\begin{equation} \label{eq:Gwrittenout}
    \begin{split}
       0 &= \langle F_n(u^*+v_n,\Kcrit+\hat{K}),v^*_n \rangle \\
       0&= P^s_n \left( F_n(u^*+v_n,\Kcrit+\hat{K})\right). 
    \end{split}
\end{equation}
Denote the system \eqref{eq:Gwrittenout} as $\mathcal{H}_n(\hat{K},v_n)$ where $\mathcal{H}_n:\mathbb{R}\times X^s_n\to \mathbb{R}\times X^s_n$. 

\begin{lem}\label{lem:maptozero} 
    There exists an $N\geq 1$ such that for any $n\geq N$ there exist $K_n^*\in \mathbb{R}$ and $\phi^s_n\in X^s_n$ such that 
    \[ 
        \mathcal{H}_n(K_n^*,\phi^s_n)=0. 
    \] 
    Additionally, it holds that 
    \begin{equation} \label{eq:pertconvinn}
        \lim_{n \to \infty}|K^*_n|= 0, \qquad \lim_{n\to \infty }\|\phi^s_n\|_{\infty}= 0.
        \end{equation}
\end{lem}

\begin{proof}
Begin by taking $N$ large enough to guarantee that the results of Lemma~\ref{lem:specDFn} holds for all $n \geq N$. Thus, we may assume for the remainder of the proof that the spectrum of $L^s_n$ is contained in the set $\{z\in\mathbb{C} \ | \ |z+\zeta|<r \}$.  This set avoids the imaginary axis and therefore $L^s_n$ is invertible as an operator on $X^s_n$.  

Consider $n$ large and fixed. We proceed as in the proof of the implicit function theorem.  Expand
    \[ 
        \mathcal{H}_n(\hat{K},\phi)= \mathcal{H}_n(0,0)+D\mathcal{H}_n(0,0)\left(\begin{array}{c} \hat{K} \\ \phi\end{array}\right)+\mathcal{N}_n(\tilde{K},\phi), \]
    where 
     \[ 
        D\mathcal{H}_n(0,0)=\left(\begin{array}{cc} \langle \mathcal{W}_n[u^*],v^*_n\rangle  & 0 \\ P^s_n\mathcal{W}_n[u^*] & L^s_n \end{array}\right), 
    \]
where we have introduced the notation $\mathcal{W}_n[u^*] = \int_0^1 W_n(x,y) \sin(u^*(y) - u^*(x))\mathrm{d}y$.
We require $D\mathcal{H}_n(0,0)$ to be invertible, and since $L^s_n$ is invertible on $X^s_n$, we see that $D\mathcal{H}_n(0,0)$ is invertible if and only if $\langle \mathcal{W}_n[u^*],v^*_n\rangle\neq 0$.  Perform the following expansion:
\[ 
    \langle \mathcal{W}_n[u^*],v^*_n\rangle=\langle \mathcal{W}[u^*],v^*\rangle+\langle \mathcal{W}_n[u^*]-\mathcal{W}[u^*],v^*_n\rangle+\langle \mathcal{W}[u^*],v^*_n-v^*\rangle, 
\]
where we have introduced the shorthand $\mathcal{W}[u^*] = \int_0^1 W(x,y) \sin(u^*(y) - u^*(x))\mathrm{d}y$ for simplicity. Note that $\mathcal{W}[u^*]=-\Kcrit^{-1}(\Omega(x) - \overline\Omega)$, while the non-degeneracy condition (\ref{eq:graphonnondegenerate}) in Hypothesis~\ref{hyp:SNgraphon} guarantees that the denominator, $\langle \Omega(x)-\overline\Omega,v^*\rangle$, is nonzero. Consequently, $\langle \mathcal{W}[u^*],v^*\rangle\neq 0$. Recalling that $0\leq W(x,y)\leq 1$, combining Lemma~\ref{lem:specDFn} and H\"older's inequality gives that for any $\varepsilon>0$ we have the bound $|\langle \mathcal{W}[u^*],v^*_n-v^*\rangle|\leq \varepsilon$ for all $n$ sufficiently large. Next, we have that 
\begin{equation} 
    \begin{split}
    \langle \mathcal{W}_n[u^*]-\mathcal{W}[u^*],v^*_n\rangle&=\int_0^1\int_0^1 [W_n(x,y)-W(x,y)]\sin(u^*(y)-u^*(x))v^*_n(x)\mathrm{d}y\mathrm{d}x  \\
    &= \int_0^1v^*_n(x)\int_0^1 [W_n(x,y)-W(x,y)]\sin(u^*(y)-u^*(x))\mathrm{d}y\mathrm{d}x 
    \end{split}
\end{equation}
The assumption that $\|d_W - d_{W_n}\|_\infty \to 0$ as $n \to \infty$ then implies that 
\begin{equation} \label{eq:WnWdegreeintegral}
    \lim_{n \to \infty}\sup_{x\in [0,1]} \left|\int_0^1 [W_n(x,y)-W(x,y)]\sin(u^*(y)-u^*(x))\mathrm{d}y\right| = 0, 
\end{equation}
as shown in \cite[Lemma~4.7]{bramburger2024persistence}. These facts combine to imply that, by perhaps taking $n$ larger than the $N$ required for Lemma~\ref{lem:specDFn}, $\langle \mathcal{W}_n[u^*],v^*_n\rangle\neq 0$ and therefore $D\mathcal{H}_n(0,0)$ is invertible for all $n$ sufficiently large.  In addition to invertibility, we require that the operator norm of $D\mathcal{H}_n(0,0)^{-1}$ is uniformly bounded in $n$. We will verify that $(L^s_n)^{-1}$ is uniformly bounded in $n$ which will then imply the same for $D\mathcal{H}_n(0,0)^{-1}$. First, note that since the spectrum of $L^s_n$ is contained in the ball  $|z+\zeta|<r$; see Lemma~\ref{lem:DFprop}, then $L^s_n$ can be inverted by Neumann series as
\[ (L^s_n)^{-1}w = \sum_{k=0}^\infty \left(\frac{L^s_n+\zeta}{\zeta}\right)^k\left( \frac{-w}{\zeta}\right). \]
We then have that 
\[ (L^s_n)^{-1}=\left[ I -(L^s-L^s_n) (L^s)^{-1} \right]^{-1} (L^s)^{-1}. \]
Since $\|L^s-L^s_n\|_{2\to 2} \to 0 $ as $n\to \infty$ then we then obtain that $\|(L^s_n)^{-1}\|_{2\to 2}$ is uniformly bounded in $n$.  Suppose for the sake of contradiction that $(L^s_n)^{-1}$ is not uniformly bounded in $n$ as an operator on $X^s_n$.  This would imply that there exists a sequence $w_n\in X^s_n$ for which $\|w_n\|_\infty=1$ but for which $\|(L^s_n)^{-1}w_n\|_\infty\to \infty$ as $n\to\infty$. Let $v_n=(L^s_n)^{-1}w_n$.  Then $\|w_n\|_2\leq \|w_n\|_\infty$ and $\|v_n\|_2$ is uniformly bounded in $n$.    Then $L^s_nv_n=w_n$ assumes the form  
\[ w_n= P^s_n\left[ \Kcrit\int_0^1 W_n(x,y)\cos\left(u^*(y)-u^*(x)\right)(v_n(y)-v_n(x))\mathrm{d}y\right]. \]
We then re-arrange to solve implicitly
\begin{eqnarray}
        v_n&= \frac{w_n}{Q_n(x)}+\frac{\Kcrit}{Q_n(x)} \int_0^1 W_n(x,y)\cos\left(u^*(y)-u^*(x)\right)v_n(y)\mathrm{d}y+\frac{1}{Q_n(x)}\int_0^1 Q_n(y)v_n(y)v_n^*(y) \mathrm{d}y  \nonumber \\
        &- \frac{\Kcrit}{Q_n(x)} \int_0^1\int_0^1 W_n(x,y)\cos\left(u^*(y)-u^*(x)\right)v_n(y)v_n^*(x)\mathrm{d}y\mathrm{d}x,
\label{eq:vnbd}    \end{eqnarray}
where  we recall that $Q_n(x)$ is the  multiplication part of the operator $L_n$ which is uniformly bounded away from zero for $n$ sufficiently large; see Lemma~\ref{lem:DFprop}.  H\"older's inequality applied to the three integrals in (\ref{eq:vnbd}) then implies that uniform boundedness of $v_n$ in $L^2([0,1])$ implies uniform boundedness in $X^s_n$.

We now consider $\mathcal{H}_n(0,0)$, given by 
\[ 
    \mathcal{H}_n(0,0)=\left(\begin{array}{c} \langle \Omega_n(x) - \omega^*_n+\Kcrit \mathcal{W}_n[u^*],v^*_n \rangle \\
    P^s_n\left(\Omega_n(x) - \omega^*_n +\Kcrit \mathcal{W}_n[u^*]\right) \end{array}\right). 
\]
Note that 
\[ 
    \begin{split}
    \   \Omega_n(x) - \omega_n^* +\Kcrit \mathcal{W}_n[u^*] &= \Omega_n(x) - \omega_n^*+\Kcrit (\mathcal{W}_n[u^*]-\mathcal{W}[u^*]) +\Kcrit \mathcal{W}[u^*] \\
    & =\Omega_n(x)-\Omega(x) + \overline\Omega - \omega_n^* +\Kcrit (\mathcal{W}_n[u^*]-\mathcal{W}[u^*]).  
    \end{split}
\]
By assumption we have that $\|\Omega - \Omega_n\|_\infty \to 0$ as $n \to \infty$, which further gives that $|\overline\Omega - \omega_n^*| \to 0$ as $n\to\infty$. Combining these facts with the estimate \eqref{eq:WnWdegreeintegral} gives  
\[ 
    \lim_{n \to \infty}\|\Omega_n(x)-\omega^*_n+\Kcrit \mathcal{W}_n[u^*]\|_\infty = 0, 
\]
which then implies that $\|\mathcal{H}_n(0,0)\|_\infty\to 0$ as $n\to\infty$.

Next, consider the operator
    \[ 
        \mathcal{T}_n(\hat{K},\phi)=-\left(D\mathcal{H}_n(0,0)\right)^{-1} \mathcal{H}_n(0,0)-\left(D\mathcal{H}_n(0,0)\right)^{-1}\mathcal{N}_n(\hat{K},\phi).
    \]
    We will show that this operator is a contraction on the Banach space $Z_n=\mathbb{R}\oplus X^s_n$, where a $z=(z_1,z_2)\in Z_n$ with $z_1\in\mathbb{R}$ and $z_2\in X^s_n$ and $Z_n$ is endowed with the norm $\|(z_1,z_2)\|_{Z_n} := \max\{|z_1|,\|z_2\|_\infty\}$.  Define
\[ 
    \mathfrak{h}_n=-\left(D\mathcal{H}_n(0,0)\right)^{-1} \mathcal{H}_n(0,0).
\]
Since $D\mathcal{H}_n(0,0)$ is boundedly invertible on $Z_n$ and $\|\mathcal{H}_n(0,0)\|_\infty \to 0$, we have that $\|\mathfrak{h}_n\|_{Z_n} \to 0 $ as $n\to\infty$ as well. 

The nonlinear terms further satisfy the following estimate 
\begin{equation} 
    \|\mathcal{N}(\hat{K},\phi)\|_\infty \leq C_1 |\hat{K}|\|\phi\|_\infty +C_2 \|\phi \|_\infty^2, \label{eq:Nestincontraction} 
\end{equation}
for some fixed  positive constants $C_1$ and $C_2$. Next, let $(\hat{K},\phi)\in B_\rho(0) = \{z\in Z_n|\ \|z\|_{Z_n} < \rho\}\subset Z_n$ for some $\rho \in (0,1)$ to be selected subsequently. Then, the estimates on the nonlinear terms give the bound 
\begin{equation} \label{eq:Tmapsballtoball}
    \begin{split}
        \| \mathcal{T}_n(\hat{K},\phi)\|_\infty &\leq \|\mathfrak{h}_n\|_\infty+\left\|D\mathcal{H}_n(0,0)^{-1}\right\|\left(  C_1 |\hat{K}|\|\phi\|_\infty +C_2 \|\phi\|_\infty^2\right) \\
        &\leq \|\mathfrak{h}_n\|_{Z_n}+\tilde{C}_1 \rho^2,
    \end{split} 
\end{equation}
for some $\tilde{C}_1>0$ independent of $\rho$ so long as $\rho < 1$.  

Now, let $z_a$ and $z_b$ be any two elements of $B_\rho(0)$.  A similar chain of reasoning to what was carried out above leads to the estimate
\begin{equation}\label{eq:Tcontraction}
    \|\mathcal{T}_n(z_a)-\mathcal{T}_n(z_b)\|_{Z_n}\leq \tilde{C}_2 \rho \|z_a-z_b\|_{Z_n},
\end{equation}
for some fixed constant $\tilde{C}_2>0$ independent of $\rho$ so long as $\rho < 1$. Taking $\rho= \min\{\frac{1}{2\max\{\tilde{C}_1,\tilde{C}_2\}},\frac{1}{2}\}$. Then since $\|\mathfrak{h}_n\|_\infty\to 0$, we find that for all $n$ taken sufficiently large, the above estimates guarantee that $\|\mathfrak{h}_n\|_\infty< \frac{\rho}{2}$ and we have that $\mathcal{T}_n:B_\rho(0)\to B_\rho(0)$ is a contraction on $Z_n$ with contraction constant at most $\frac{1}{2}$.  Therefore, for all $n$ taken sufficiently large there exists a unique fixed point $(\tilde{K}^*_n,\phi^*_n)\in B_\rho(0)$.

We conclude by establishing the convergence facts stated in \eqref{eq:pertconvinn}.  A priori, the contraction mapping theorem only states that these fixed points lie in $B_\rho(0)$.  However, if we define $\mathfrak{s}_n=\mathcal{T}_n\mathfrak{h}_n$ we can combine the fact that $\|\mathfrak{h}_n\|_{Z_n}\to 0$ with the bounds on the nonlinear term in \eqref{eq:Nestincontraction} to find that $\|\mathfrak{s}_n\|_{Z_n}\to 0$ as well. Then a corollary of the contraction mapping theorem implies that the fixed point $(\tilde{K}^*_n,\phi^*_n)$ satisfies $\|(\tilde{K}^*_n,\phi^*_n)-\mathfrak{s}_n\|_{Z_n} \leq \frac{\kappa}{1-\kappa} \|\mathfrak{h}_n-\mathfrak{s}_n\|_{Z_n}$ where $\kappa$ is the contraction constant associated to the contraction mapping.  We therefore obtain 
\[ 
    \| (\tilde{K}^*_n,\phi^*_n)-\mathfrak{s}_n\|_{Z_n}\leq \frac{\frac{1}{2}}{1-\frac{1}{2}} \|\mathfrak{h}_n-\mathfrak{s}_n\|_{Z_n}
\]
which in turn can be rearranged to find that
\[
    \| (\tilde{K}^*_n,\phi^*_n)\|_{Z_n} \leq \|\mathfrak{h}_n\|_{Z_n} + 2\|\mathfrak{s}_n\|_{Z_n}.    
\]
Since we have already established that $\|\mathfrak{h}_n\|_{Z_n},\|\mathfrak{s}_n\|_{Z_n} \to 0$ as $n \to \infty$, we arrive at the results \eqref{eq:pertconvinn}. This completes the proof of the lemma.
\end{proof} 

We now consider the following transformations
\begin{equation} \label{eq:transformation}
    \begin{split}
        v_n&= \phi^s_n+\tilde{v}_n \\
        K &= \Kcrit +K^*_n+\tilde{K}.  
    \end{split}
\end{equation}
Define $H_n:X_n\times \mathbb{R}\to X_n$ so that $\tilde{v}_n$ solves \eqref{eq:Heqn} with
\[ 
    H_n(\tilde{v}_n,\tilde{K})=\Omega_n - \omega_n^* +(\Kcrit +K^*_n+\tilde{K})\mathcal{W}_n\left[u^*+\phi^s_n+\tilde{v}_n \right],  
\]
where we continue with the notation $\mathcal{W}_n[u] = \int_0^1 W_n(x,y)\sin(u(y) - u(x))\mathrm{d}y$ introduced in the previous proof. According to Lemma~\ref{lem:maptozero} it follows that $H_n(0,0)=0$.  Moreover, 
\begin{equation}
    \begin{split}
        D_{\tilde{v}_n}H_n(0,0)&= (\Kcrit+K^*_n)D\mathcal{W}_n[u^*+\phi^s_n] \\
        D_{\tilde{K}}H_n(0,0)&=\mathcal{W}_n[u^*+\phi^s_n].
    \end{split}
\end{equation}

\begin{lem} \label{lem:tildelambda} 
Let $\tilde{L}_n=D_{\tilde{v}_n}H_n(0,0)$. For all $\varepsilon > 0$ there exists an $\tilde{N} \geq 1$ such that for all $n \geq N$ the linear operator $\tilde{L}_n:X_n\to X_n$ has a simple eigenvalue $\tilde{\lambda_n}$ with $|\tilde{\lambda}_n| < \varepsilon$ and associated eigenfunction $\tilde{v}^*_n(x) \in X_n$, normalized so that $\langle \tilde{v}^*_n,\tilde{v}^*_n\rangle=1 $, satisfying 
\begin{equation}\label{eq:vntildevnstarestimate} 
    \|v_n^*-\tilde{v}_n^*\|_\infty < \varepsilon,  
\end{equation}
Furthermore, the remainder of the spectrum lies in the ball 
\[
    \{z\in\mathbb{C} \ | \ |z+\zeta|<r\}
\] 
and the spectral projection onto the eigenspace of the isolated eigenvalue $\tilde{\lambda}_n$ is 
\[ 
    \tilde{P}^c_n f=\tilde{v}^*_n(x)\int_0^1 f(y) \tilde{v}_n^*(y)\mathrm{d}y=\tilde{v}^*_n(x)\langle f, \tilde{v}^*_n\rangle. 
\]
A stable projection is defined via $\tilde{P}^s_n=I-\tilde{P}^c_n$.
\end{lem}

\begin{proof}
Write 
\[ 
    \tilde{L}_n=L_n +K_n^* D\mathcal{W}_n[u^*+\phi^s_n] +\Kcrit \left(D\mathcal{W}_n[u^*+\phi^s_n]-D\mathcal{W}_n[u^*]\right).  
\]
By \eqref{eq:pertconvinn} it follows that 
\[ 
    \|\tilde{L}_n-L_n\|_{\infty\to\infty} \to 0, 
\]
as $n\to\infty$. The result then follows from spectral convergence results as in \cite[Theorem IV.3.1 and Theorem IV.3.16]{kato}.
\end{proof}

We now decompose the solution into 
\[
    \tilde{v}_n(t,x)=w^c_n(t) \tilde{v}^*_n(x)+\tilde{v}^s_n(t,x),
\]
where $w^c_n\in\mathbb{R}$ and $\tilde{v}^s_n\in \tilde{X}^s_n$.

\begin{prop}\label{prop:stepgraphonCM}  
There exists an $N\geq 1$ and an $\varepsilon>0$ such that for any  $n\geq N$ there exists 
\begin{enumerate}
    \item open neighborhoods $B^c_\varepsilon(0)\subset \tilde{X}^c_n$ and   $B^s_\varepsilon(0)\subset \tilde{X}^s_n$,  
    \item an open interval $I_K=(-\varepsilon,\varepsilon)$, and 
    \item for any $k>2$, a $C^k$ mapping $\Psi_n:\mathbb{R}\times \mathbb{R}\times \mathbb{R} \to \tilde{X}^s_n$, 
\end{enumerate}
such that the manifold
\begin{equation}\label{MnCenterManifold}
    \mathcal{M}_n=\left\{ w^c_n \tilde{v}^*_n+\Psi_n(w^c_n,\tilde{K},\tilde{\lambda}_n)\  | \  w^c_n\in \mathbb{R} \right\},
\end{equation}
is a locally invariant center manifold.  For any fixed $|\tilde{K}|<\varepsilon$, $\mathcal{M}_n$ contains all solutions in $B^c_\varepsilon(0)\times B^s_\varepsilon(0)$ which remain bounded for all time. 

Moreover, we obtain the following facts regarding the reduced flow on the center manifold and the description of the manifold itself:
\begin{itemize}
\item[i)] The center manifold admits the following expansion 
\begin{equation}
\Psi_n(w^c_n,\tilde{K},\tilde{\lambda}_n)=\Psi_{n,010}\tilde{K}+\mathcal{O}\left(\left(w^c_n+\tilde{K}+\tilde{\lambda}_n\right)^2\right) ,
\end{equation}
for some $\Psi_{n,010}\in \tilde{X}^s_n$.
\item[ii)] The reduced equation on the center manifold assumes the form
\begin{equation} \label{eq:reducedCM}
    \frac{dw^c_n}{dt}=\tilde{\lambda}_nw^c_n+a_n\tilde{K} +b_n (w^c_n)^2+\mathcal{O}(w^c_n\tilde{K},w^c_n\tilde{\lambda}_n,\tilde{K}^2,|w^c_n+\tilde{K}+\tilde{\lambda}_n|^3)
\end{equation}
where 
\begin{equation}\label{eq:anbns}
    \begin{split}
        a_n &= \int_0^1 \int_0^1 W_n(x,y)\sin\left(u^*(y)+\phi^s_n(y)-u^*(x)+\phi^s_n(x)\right)\tilde{v}^*_n(x)\mathrm{d}y\mathrm{d}x\\
        b_n &=-\frac{\Kcrit+K^*_n}{2} \int_0^1\int_0^1 W_n(x,y)\sin\left(u^*(y)+\phi^s_n(y)-u^*(x)+\phi^s_n(x)\right)\left(\tilde{v}^*_n(y)-\tilde{v}^*_n(x)\right)^2\tilde{v}^*_n(x)\mathrm{d}y\mathrm{d}x.
    \end{split}
\end{equation}
  Furthermore, $a_n\to a$ and $b_n\to b$ as $n\to\infty$.
\end{itemize}
\end{prop}
\begin{proof} The function $H_n$ is a $C^k$ mapping for any $k>0$ so the existence of a $C^k$ center manifold follows  \cite[Theorem 3.3]{haragus11}.  The proof is presented in Appendix~\ref{sec:stepgraphonproof}. 
\end{proof}

\begin{cor}
    For all sufficiently large $n$, the discrete equation \eqref{eq:mainstepgraphon} has a saddle-node bifurcation at coupling parameter $\Kncrit$ satisfying
    \[ 
        \lim_{n \to \infty }|\Kncrit-\Kcrit| = 0.
    \]
\end{cor}

\begin{proof} 
We examine the reduced flow within the center manifold given in \eqref{eq:reducedCM}. Recall from Lemma~\ref{lem:tildelambda} that for $n \gg 1$ the eigenvalue $\tilde{\lambda}_n$ is small. Thus, we take $n$ sufficiently large to apply the results of our previous findings and then introduce the rescalings
\[ 
    \lambda_n = \eta\mu, \qquad w^c_n=\eta z, \qquad \tilde{K}=\eta^2\kappa, 
\]
where $|\eta|$ is a small quantity and we neglect the dependence of $(\mu,z,\kappa)$ on $n$ to simplify the presentation. With this rescaling, equilibrium solutions on the center manifold \eqref{eq:reducedCM} satisfy
\[ 
    0=\eta^2\left( \mu z + a_n\kappa + b_nz^2\right) +\mathcal{O}(\eta^3) .
\]
Upon dividing through by $\eta^2$, solving the leading order quadratic equation $\mu z+a_n\kappa +b_nz^2 = 0$ gives the existence of two branches of equilibria which coalesce at 
\[ 
    z^*=-\frac{\mu}{b_n}, \qquad \kappa =\frac{\mu^2}{4a_nb_n}.
\]
The implicit function theorem allows one to smoothly perturb this critical point in $\eta$ about $\eta = 0$, and so reverting to the original coordinates we obtain the existence a saddle-node bifurcation in \eqref{eq:reducedCM} occurring at 
\[ 
    \tilde{K}=\frac{\tilde{\lambda}_n^2}{4a_nb_n}+\mathcal{O}(\tilde{\lambda}_n^3). 
\]
Reverting further to \eqref{eq:transformation} we obtain
\begin{equation}\label{KncritExact}
    \Kncrit=\Kcrit+K^*_n+\frac{\tilde{\lambda}_n^2}{4a_nb_n}+\mathcal{O}(\tilde{\lambda}_n^3). 
\end{equation}
Since we have $K_n^*,\tilde{\lambda}_n \to 0$ as $n\to \infty$ via Lemmas~\ref{lem:maptozero} and \ref{lem:tildelambda}, respectively, we have now proven all statements the corollary. 
\end{proof}

The result of the work in this section is that we have demonstrated that under the assumptions of Theorem~\ref{thm:Bifurcation} we have shown that a saddle-node bifurcation also takes place in the step function model \eqref{eq:mainstepgraphon}. Moreover, this bifurcation takes place at $\Kncrit$, as given in \eqref{KncritExact} which asymptotically approaches $\Kcrit$ as $n \to \infty$. Note however that we have not proven Theorem~\ref{thm:Persistence} in its entirety yet though since we have not returned to the finite-dimensional model $G_n$ in \eqref{GFunction}, representing the right-hand-side of the Kuramoto model \eqref{Kuramoto}. This final step is taken care of in the following subsection.

\subsection{Finite-Dimensional Solutions}\label{sec:PiecewiseConstant}

Up to this point, we have only dealt with the step graphon equation \eqref{eq:mainstepgraphon}, whose steady-states correspond to solving $F_n(u_n,K) = 0$ in \eqref{FnFunction}. As was discussed in Section~\ref{sec:Results}, if we have that some pair $(u^*_n,K^*)\in X_n\times \mathbb{R}$ solves $F_n(u^*_n,K^*) = 0$ with $u^*$ piecewise constant over the intervals $\{I_j^n\}_{j = 1}^n$, then $F_n = 0$ completely reduces to the finite-dimensional problem $G_n = 0$ in \eqref{GFunction} at the same $K = K^*$. Thus, here we provide a result showing that under the assumptions of Proposition~\ref{prop:stepgraphonCM} all steady-state solutions on the center manifold $\mathcal{M}_n$, as defined in \eqref{MnCenterManifold}, are piecewise constant over the intervals $\{I_j^n\}_{j = 1}^n$. This in turn will complete the proof of Theorem~\ref{thm:Bifurcation} as it brings us back to the finite-dimensional Kuramoto model.
 
\begin{lem}\label{lem:PiecewiseConstant}
    Let $N \geq 1$ and $\varepsilon > 0$ be as guaranteed by Proposition~\ref{prop:stepgraphonCM}. Then, for every $n\geq N$, every steady-state solution on the center manifold $\mathcal{M}_n$ in \eqref{MnCenterManifold} is piecewise constant over the intervals $\{I_j^n\}_{j = 1}^n$.
\end{lem}

\begin{proof}
The proof of this result is exactly the same as that of \cite[Lemma~4.18]{bramburger2024persistence}, and so here we only sketch out the details at a high level for the reader. First, the linearization of $F_n$, denoted $DF_n$, about any root $u^*_n \in X_n$ at a fixed value of $K$ is broken up into two pieces: a nonlocal Hilbert--Schmidt integral operator and a multiplication operator. The multiplication operator $v(x) \mapsto Q_n(x)v(x)$ takes the form 
\[
    Q_n(x) = -K\int_0^1 \cos(u_n^*(y) - u_n^*(x))\mathrm{d}y.
\]
Further, the spectrum of $DF_n(u^*_n,K)$ is broken into disjoint sets defined as those $\lambda\in\mathbb{C}$ for which $DF_n(u^*_n,K) - \lambda$ is a non-invertible Fredholm operator (the point spectrum) and when it is not a Fredholm operator (the essential spectrum). The result \cite[Lemma~4.1]{bramburger2024persistence} proves that the essential spectrum is exactly equal to the range of $Q_n$. Moreover, Lemma~\ref{lem:tildelambda} proves that the essential spectrum must be confined to the left half of the complex plane for any steady-state solution $(u^*_n,K)\in\mathcal{M}_n$, giving that $Q_n(x) < 0$ for all $x\in [0,1]$ and $n \geq N$, and in particular $Q_n(x) \neq 0$ everywhere. One then achieves the proof of this result by  assuming that for some fixed $n$ the solution $u_n^*$ is non-constant over one of the subintervals $\{I_j^n\}_{j = 1}^n$. The contradiction is reached by showing that this would imply that $Q_n(x) = 0$ for some $x$, which we have already argued cannot happen. 
\end{proof} 

With Lemma~\ref{lem:PiecewiseConstant} we have now proven Theorem~\ref{thm:Bifurcation} in its entirety. The most important takeaway from the sketch of the proof above is that it is only the stability of the essential spectrum that is used to show that solutions of $F_n = 0$ are piecewise constant. Thus, our results could be applied more broadly to capture other bifurcations in random Kuramoto networks, as well as prove the existence of higher-dimensional center manifolds using only the nonlocal graphon model \eqref{KuramotoGraphon}.

\section{Comments on the Proof of Theorem~\ref{thm:Persistence}}\label{sec:PersistenceProof}

We provide only a brief commentary on the proof of Theorem~\ref{thm:Persistence}. This is because it is can be seen as an application of our previous result \cite[Theorem~3.1]{bramburger2024persistence}. Alternatively, one can arrive at the proof following in a manner similar to that of Theorem~\ref{thm:Bifurcation} in the previous section. Precisely, the proof of Theorem~\ref{thm:Persistence} is similar to that of Lemma~\ref{lem:maptozero} in a simplified setting. This is because according to the assumptions of Theorem~\ref{thm:Persistence} the spectrum $DF(u^*,K)$ is bounded away from the imaginary axis, meaning that there is no need to divide $X_n$ since $X_n^c = \emptyset$ in this case. This means that we need only solve $F_n(u^* + v_n,K) = 0$ for $v_n \in X_n$, while the linearization $DF_n(u^*,K)$ is boundedly invertible on $X_n$. The existence of such a $v_n$ is obtained with a nearly identical application of the contraction mapping theorem, but now keeping in mind that $K$ is fixed, thus simplifying the problem slightly. Finally, an identical result to Lemma~\ref{lem:PiecewiseConstant} will show that the solution $u^* + v_n \in X_n$ is piecewise constant over the intervals $\{I_j^n\}_{j = 1}^n$ since the essential spectrum is again bounded away from the imaginary axis.

\section{Discussion}\label{sec:Discussion}

In this paper we have developed a framework for characterizing both the onset and persistence of synchronous solutions to random Kuramoto models through the study of a single master nonlocal equation. The result is an applicable way of studying random networks of coupled oscillators to predict both when synchronous solutions exist and what they look like. A major application of our work herein was to Erd\H{o}s--R\'enyi networks in Section~\ref{sec:ApplicationER}, where we leveraged and extended Ermentrout's pioneering work in \cite{ermentrout1985synchronization}. With this application we also saw that bifurcations to synchrony in the graphon equation do not always come in the form of a standard saddle-node bifurcation, thus rendering our Theorem~\ref{thm:Bifurcation} inapplicable in this scenario. In particular, we were able to prove that in some cases the onset of synchrony comes from bifurcations involving the essential spectrum, a situation that warrants a follow-up investigation. Interestingly, we saw that while bifurcations from the essential spectrum may violate our theoretical assumptions, the results still seem to hold in that both the critical coupling point and the shape of the synchronous solutions are predicted by the master graphon equation. 

While the application in this manuscript was to coupled oscillators, we believe that they are broadly applicable to patterns and oscillations over a variety of randomly networked dynamical systems. That is, Theorem~\ref{thm:Persistence} is mostly a particular instantiation of the previous work \cite{bramburger2024persistence} which used nonlocal graphon models to predict the existence of steady-states to dynamical systems on networks. Thus, it seems reasonable to expect that our bifurcation results from Theorem~\ref{thm:Bifurcation} could similarly be extended to more general networked dynamical systems, as well as other co-dimension one bifurcations. Moreover, extending our center manifold results (see Lemma~\ref{lem:CMgraphon} and Proposition~\ref{prop:stepgraphonCM}) to more general networked systems would provide a method of obtaining invariant manifolds for dynamical systems on random networks.  

\begin{figure} 
    \center
    \includegraphics[width = 0.45\textwidth]{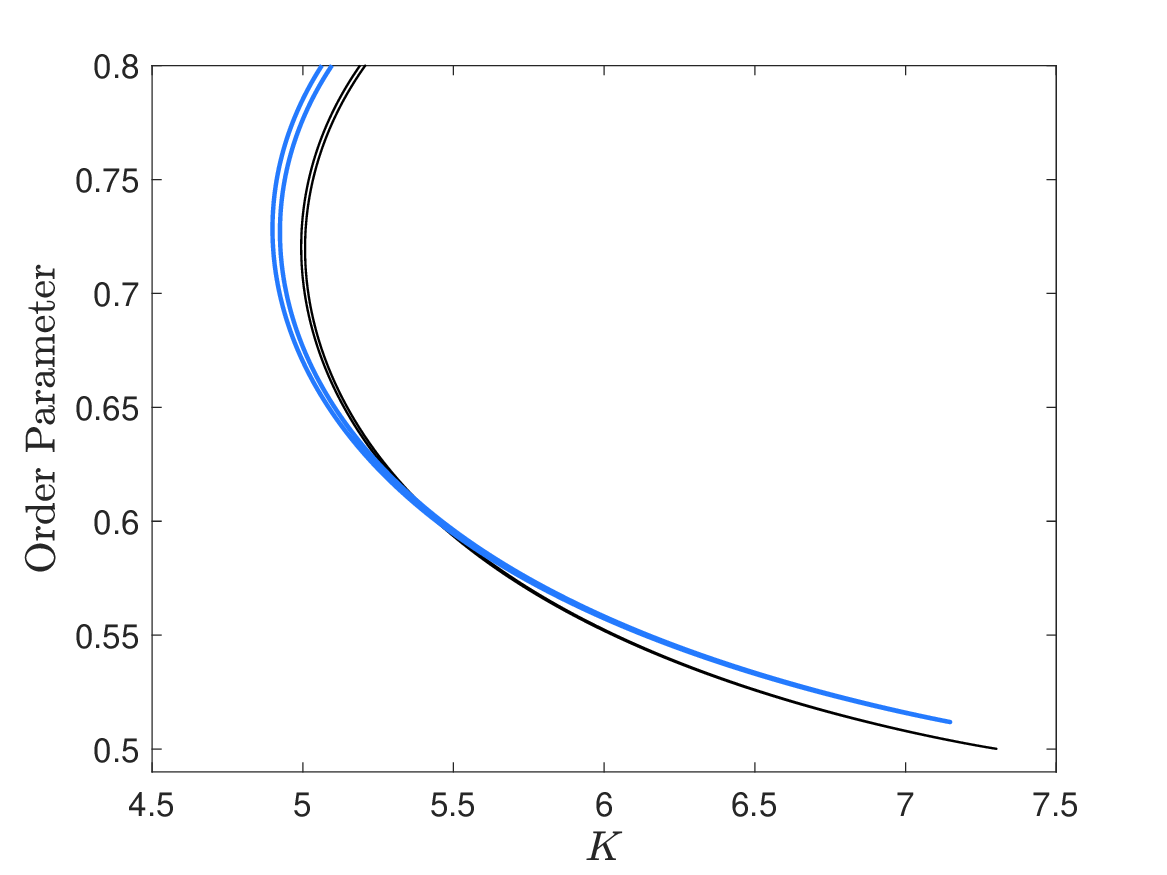} 
    \includegraphics[width = 0.45\textwidth]{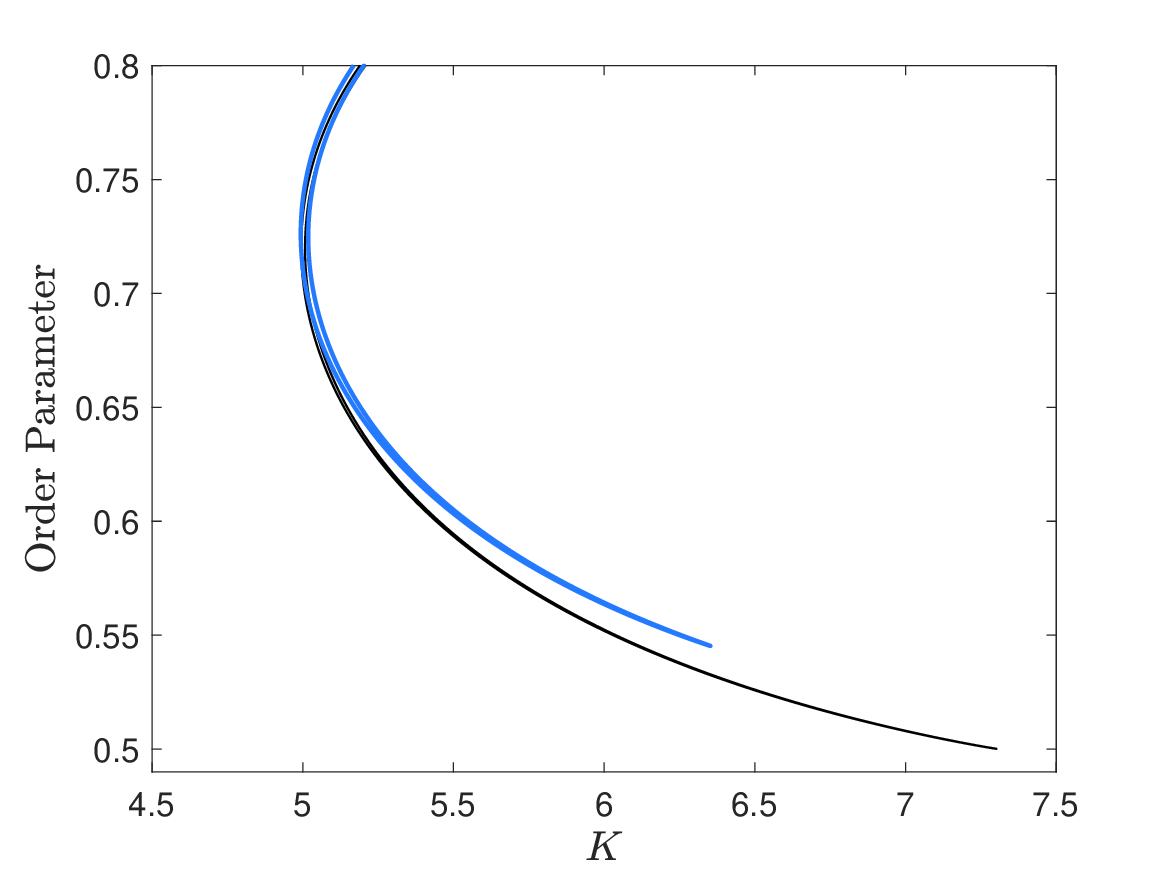}   
    \caption{Bifurcation diagrams comparing random small-world Kuramoto networks (blue) with the graphon model (black) as the order parameter \eqref{OrderParam} versus the coupling coefficient $K$. The Kuramoto networks have $n = 1000$ oscillators arranged over a random weighted graph $\mathbb{H}(1000,W)$ (left) and a random simple graph $\mathbb{G}(1000,W)$ (right).}
    \label{fig:SmallWorldBif}
\end{figure}

While the extension of our results to other dynamical systems is interesting, there still remains much to work on the coupled oscillator models considered herein. That is, our applications exclusively focused on Erd\H{o}s--R\'enyi networks, but our theoretical results can be applied to a much wider range of graphon families. This therefore could open up the study of \eqref{KuramotoGraphon} with different graphons to better understand synchronization in \eqref{Kuramoto} when both the frequencies and the network are random. For example, in Figure~\ref{fig:SmallWorldBif} we provide the bifurcation diagram (in black) for the graphon model \eqref{KuramotoGraphon} with the small-world graphon   
\begin{equation}\label{SmallWorldGraphon}
    W(x,y) = \begin{cases}
        0.9 & \min\{1 - |x-y|,|x - y|\} \leq 0.25 \\
        0.1 & \mathrm{otherwise},
    \end{cases}
\end{equation}
and frequencies drawn from the distribution \eqref{CosineDistribution}. We see what appears to be three distinct saddle-node bifurcations at $K = 4.99, 5.01, 7.30$, all of which could be applicable to our analysis. For comparison, we further provide continuations of synchronous solutions (in blue) for the Kuramoto model \eqref{Kuramoto} on random weighted $\mathbb{H}(1000,W)$ and simple $\mathbb{G}(1000,W)$ graphs (using the notation of Section~\ref{sec:Graphons}) of $n = 1000$ oscillators. While the finite-dimensional saddle-node bifurcations are close to the graphon model prediction, much analysis is required to demonstrate that our theorems apply to the results in the figure. Precisely, (i) are the bifurcations a standard saddle-node or something more complicated from the essential spectrum, and (ii) can we prove that all bifurcating solutions to the graphon model are continuous? Such questions and applications remain for a follow-up investigation.

Finally, we note that we have focused on graphons defined on the canonical space $[0,1]$.  Extensions to other probability spaces is feasible, as laid out in \cite{janson2010graphons}. In particular, we expect that our results can be extended to graphons taking the form $W(\mathbf{x},\mathbf{y})$ where $\mathbf{x}=(x_1,x_2) \in [0,1]\times[0,1]$ with the natural frequency of an oscillator given by $\Omega(x_1)$, while the probability of a connection between an oscillator with latent position $(x_{1,j},x_{2,j})$ and $(x_{1,k},x_{2,k})$ is given by $W(\mathbf{x}_j,\mathbf{x}_k)=W(x_{2,j},x_{2,k})$ thus decorrelating the intrinsic frequencies of each oscillator and their network structure. We expect this to be a straightforward extension of the work herein, but leave the details to a follow-up investigation.

\section*{Acknowledgments}

JJB was partially supported by an NSERC Discovery Grant through grant RGPIN-2023-04244 and the Fondes de Recherche du Qu\'ebec – Nature et Technologies (FRQNT) through grant 340894.  MH  was partially supported by the National Science Foundation through DMS-2406623.

\begin{appendix}

\section{Proof of Lemma~\ref{lem:OmegaConv}}\label{app:OmegaConv}

Begin by assuming that $\Omega:[0,1] \to [-1,1]$ is a continuous function and for each $n \geq 1$ let $ \{x_1,x_2,\dots,x_n\}$ be an ordered $n$-tuple of independent uniform random points drawn from $[0,1]$. Define the empirical distribution function
\[ 
    F_n(x)=\frac{1}{n}\sum_{j=1}^n \chi_{(x_j,\infty)}(x),  
\]
where $\chi_S(x)$ is the indicator function associated to the set $S$. We further consider the generalized inverse of $F_n$, the empirical quantile function, given by
\[ 
    G_n(x)= x_j \ \text{for} \ x\in \left[\frac{j-1}{n},\frac{j}{n}\right). 
\] 
Note that the Glivenko--Cantelli Theorem \cite[Theorem 20.6]{billingsley} implies that 
\[ 
    \|F_n-F\|_\infty=\sup_x \left| F_n(x)-x\right| \to 0 
\]
as $n \to \infty$ almost surely, where $F(x) = x$ is the cumulative distribution function for the uniform distribution on $[0,1]$. We will now show that $\|G_n - G\|_\infty\to 0$ as $n \to \infty$ almost surely as well, where $G(y) = F^{-1}(y) = y$.

Letting $\varepsilon > 0$, there exists $N \geq 1$ so that for all $n \geq N$ we have $\|F_n - F\|_\infty < \varepsilon$ with probability 1. Now, suppose that for some $n \geq N$ there exists a $y_n$ such that $|G_n(y_n) - y_n| > \varepsilon$. We will show that this is a probability 0 event when $n \geq N$, which in turn shows that $\|G_n - G\|_\infty \to 0$ as $n \to \infty$ almost surely. Indeed, there exists a $j \in \{1,\dots,n\}$ so that $y_n\in I_j^n$ and monotonicity of $G(y) = y$ implies that the maximal deviation between $G_n$ and $G$ in $I_j^n$ occurs at an endpoint of this interval. At the left endpoint $\zeta_\ell = (j-1)/n = F(x_j)$, we use the fact that $G_n(\zeta) = x_j$ to get 
\[ 
    |G_n(\zeta)-\zeta|=|x_j-F_n(x_j)|>\varepsilon.
\]
The equality above shows that $|G_n(\zeta)-\zeta| > \varepsilon$ happens with probability 0 for $n \geq N$, showing that $|G_n(y_n) - y_n| \leq |G_n(\zeta)-\zeta| < \varepsilon$ with probability 1. 

The above shows that $\|G_n - G\|_\infty \to 0$ as $n\to\infty$ almost surely. Finally, note that $\Omega_n(x) = \Omega(G_n(x))$ for each $n \geq 1$. Since $\Omega$ is assumed continuous, it follows that $\|\Omega_n - \Omega\|_\infty = \|\Omega\circ G_n - \Omega\circ G\|_\infty \to 0$ as $n\to\infty$ almost surely, completing the proof.

\section{Proof of Lemma~\ref{lem:CMgraphon}}\label{sec:graphonproof}

We verify the hypothesis of the parameter-dependent center manifold theorem; see  \cite[Theorem 3.3]{haragus11}. Let $u(t,x)=u^*(x)+v(t,x)$ and $K=\Kcrit+\tilde{K}$ in \eqref{eq:maingraphon}. Then, 
\begin{equation} \label{eq:graphonsetupforCM}
    \frac{dv}{dt}=F[u^*+v,\Kcrit + \tilde{K}].
\end{equation}
For the ease of presentation, let us denote $H(v,\tilde{K})=F[u^*+v,\Kcrit + \tilde{K}]$. Note that $H$ is smooth in its arguments, $H(0,0)=0$, and $D_vH(0,0)$ is the linear operator $L$ described previously. Thus, equation \eqref{eq:graphonsetupforCM} assumes the form required for an application of Theorem 3.3 of \cite{haragus11}. The spectral properties required for an application of this result are spelled out in the following lemma. 

\begin{lem} \label{lem:DFprop}
    The linearization $L:X\to X$ has the following properties: 
    \begin{itemize}
    \item[i)] The spectrum $\sigma(L)$ as an operator on $X$ has the decomposition $\sigma=\sigma_0\cup \sigma_s$ where $\sigma_0=\{0\}$ and there exist an $\alpha>0$ such that 
    \[ \sup_{\lambda\in \sigma_s} \mathrm{Re}(\lambda)<-\alpha. \]
        \item[ii)] Restricted to the class of continuous mean-zero functions $X$, the algebraic multiplicity of zero as an eigenvalue of $L$ is one and there exists a function $v^*\in X$ normalized such that the (central) spectral projection $P^c:X \to X$ has the following representation 
        \[ 
            P^cf = v^*(x)\int_0^1 f(y)v^*(y)\mathrm{d} y. 
        \]
        \item[iii)] There exists a $\xi>0$ and $r<\xi-\alpha$  such that 
        \[ \sigma_s\subset \left\{ z\in \mathbb{C} \ | \ |z+\xi|< r \right\},  \]
       with (stable) spectral projection $P^s=I-P^c$.
       \item[iv)] Let $X^s=\mathrm{Rng} P^s$ and define $L^s=L|_{X^s}$.  Then $L^s$ generates an analytic semigroup on $X^s$, which we denote $e^{L^st}$.  Moreover, the following estimate holds
       \begin{equation} 
        \|e^{L^s t} \|\leq C e^{-\alpha t},
        \label{eq:Lstemp} \end{equation}
        for any $t>0$.
       \item[v)]  Let $f\in C_\eta(\mathbb{R},X^s)$ where $\|f\|_\eta=\sup_{t\in\mathbb{R}}\left(e^{-\eta |t|}\|f(t,\cdot)\|_{\infty}\right) $
       Then 
       \[ 
        \frac{dv_s}{dt}=L^sv_s+f(t), 
       \]
       has a unique solution given by
       \[ 
        v_s(t)=\int_{-\infty}^t e^{(t-\tau)L^s} P^s f(\tau) \mathrm{d} \tau.  
       \]
       Furthermore, there exist a continuous function $\kappa(\eta)$ such that $\|v_s\|_{C_\eta}\leq \kappa(\eta)\|f\|_{C_\eta}$ for $\eta$ sufficiently small.
    \end{itemize}
\end{lem}
\begin{proof} 
Conclusion (i) is simply a re-statement of part (i) of Hypothesis~\ref{hyp:SNgraphon}.  Hypothesis~\ref{hyp:SNgraphon} also gives that the zero eigenvalue is simple by assumption. Then (ii) follows from self-adjointness of the operator $L = DF(u^*,\Kcrit)$ which implies that the eigenfunction and adjoint eigenfunction are identical. Since $0\leq W(x,y)\leq 1$, the linear operator $L$ is bounded and therefore the stable spectrum is contained in a ball as stated in (iii).  Recall that the absence of unstable spectrum was also assumed in Hypothesis~\ref{hyp:SNgraphon}.  As a result of (iii), the operator $L^s$ is bounded and its spectrum is separated from the imaginary axis and can be contained in a sector.  The resolvent operator can be constructed by Neumann series:
\[ (L^s-\lambda)^{-1}=\frac{-1}{(\lambda+\xi)}\sum_{k=0}^\infty  \left(\frac{L^s+\xi}{\lambda+\xi}\right)^k. \]
There exists a $\varphi\in \left(\frac{\pi}{2},\pi \right)$ and a constant $\Xi>0$ such that for any $\lambda$ satisfying $|\mathrm{arg}(\lambda+\alpha)|<\varphi$ the following resolvent estimate holds:
\[ \left\|(\lambda+\alpha) \left(L^s-\lambda\right)^{-1}\right\|_{\infty\to\infty}\leq \left\| (\lambda+\alpha)\sum_{k=0}^\infty  \left(\frac{L^s+\xi}{\lambda+\xi}\right)^k \frac{1}{(\lambda+\xi)}\right\|_{\infty\to\infty}\leq \Xi.  \]
 Thus, the linear operator $L^s$ is sectorial and the existence of an analytic semigroup obeying the temporal bound (\ref{eq:Lstemp}) follows from standard arguments.  These estimates can be used to verify v), with $\kappa(\eta)=\frac{2C\alpha}{\alpha^2-\eta^2}$ and $\eta<\alpha$.  We omit the details of this calculation. 
\end{proof}

Lemma~\ref{lem:DFprop} confirms the necessary hypotheses to apply \cite[Theorem~3.3]{haragus11} and therefore provides the existence of a center manifold to our graphon model. This center manifold can be written as the graph
\[ 
    \mathcal{M}=\left\{ w^c v^*+\Psi(w^c,\tilde{K})\  | \  w^c\in \mathbb{R} \right\}, 
\]
where $\Psi:\mathbb{R}\times\mathbb{R}\to X^s$ is $C^k$ in its arguments for any $k>2$. The manifold is invariant and contains all solutions that remain locally bounded for all $t\in\mathbb{R}$. The reduced equation on the center manifold is obtained by
\begin{equation} 
    \frac{dw^c}{dt} = \langle F(u^*+w^cv^*+\Psi(w^c,\tilde{K}),\Kcrit+\tilde{K}),v^*\rangle.  \label{eq:graphonprojectCM} \end{equation}
We next expand 
\begin{equation} \label{eq:expandedrhswithR}
    \begin{split}
F(u^*+w^cv^*+\Psi(w^c,\tilde{K}),\Kcrit+\tilde{K}) &= L \Psi + \frac{\tilde{K}}{\Kcrit} L\Psi +\tilde{K}\int_0^1 W(x,y)\sin(u^*(y) - u^*(x)) \mathrm{d}y \\ &\quad +(\Kcrit+\tilde{K}) R(w^cv^*+\Psi(w^c,\tilde{K})),
    \end{split}
\end{equation}
where we introduce $R(w^cv^*+\Psi(w^c,\tilde{K}))$ as a remainder term to capture all higher-order terms in the expansion. Letting $h_c(w^c,\tilde{K})$ denote the right hand side of \eqref{eq:graphonprojectCM}, we note that $h_c(w^c,\tilde{K})=a\tilde{K}+\mathcal{O}(2)$ where $a=-\frac{1}{\Kcrit}\langle \Omega - \overline\Omega,v^*\rangle=\int_0^1 W(x,y)\sin(u^*(y)-u^*(x))\mathrm{d}y$. Then, the mapping $\Psi$ satisfies the invariance condition 
\begin{equation} \label{eq:graphonreductionequation}
    \left(D_{w^c}\Psi\right) h_c(w^c,\tilde{K})=P^s \left( F(u^*+w^cv^*+\Psi(w^c,\tilde{K}),\Kcrit+\tilde{K})\right).
\end{equation}

To obtain an expansion for $\Psi$ we begin with the linear ansatz:
\[ 
    \Psi(w^c,\tilde{K})=\Psi_{10}w^c+\Psi_{01}\tilde{K}+\mathcal{O}(2),  
\]
where $\Psi_{10}$ and $\Psi_{01}$ are elements of $X^s$. Substituting this first-order expansion into \eqref{eq:graphonreductionequation} and retaining only linear terms, we obtain the solvability condition  
\[ 
    a\Psi_{10}\tilde{K} = w^c L^s \Psi_{10}+ \tilde{K} L^s \Psi_{01} +\tilde{K} P^s\bigg(\int_0^1 W(x,y)\sin(u^*(y) - u^*(x)) \mathrm{d}y\bigg),   
\]
from which we obtain $\Psi_{10}=0$ and
\[
    \Psi_{01}=- (L^s)^{-1}P^s\bigg(\int_0^1 W(x,y)\sin(u^*(y) - u^*(x)) \mathrm{d}x\bigg).  
\]
We now compute higher-order expansions for the reduced equation on the center manifold. With the above determined linear terms in $\Psi$, we now obtain
\begin{equation}
    \begin{split} 
        h_c(w^c,\tilde{K})&=\langle F(u^*+w^cv^*+\Psi(w^c,\tilde{K}),\Kcrit+\tilde{K}),v^*\rangle \\
            &= \langle \tilde{K}\int_0^1 W(x,y)\sin(u^*(y) - u^*(x)) \mathrm{d}y  +(\Kcrit+\tilde{K}) R(w^cv^*+\Psi(w^c,\tilde{K})),v^*\rangle 
    \end{split}
\end{equation}
Since $\Psi$ lacks linear terms in $w^c$ and $R$ is quadratic in its argument we obtain that the reduced equation on the center manifold has the expansion
\[ 
    \frac{dw^c}{dt}=a\tilde{K}+b(w^c)^2 +\mathcal{O}\left(w^c\tilde{K},\tilde{K}^2,|w^c+\tilde{K}|^3 \right),
\]
where $b$ is given in \eqref{eq:aandb}.  This concludes the proof of Lemma~\ref{lem:CMgraphon}.

\section{Proof of Proposition~\ref{prop:stepgraphonCM}}\label{sec:stepgraphonproof}

We will apply the center manifold theorem to the system of 
\[ \frac{d\tilde{v}_n}{dt}= H_n(\tilde{v}_n,\tilde{K}), \]
where we recall
\begin{equation} \label{eq:HNexp}
\begin{split}
    H_n(\tilde{v}_n,\tilde{K})&=\Omega_n - \omega_n^* +(\Kcrit +K^*_n+\tilde{K})\mathcal{W}_n\left[u^*+\phi^s_n+\tilde{v}_n \right] \\
    &= \tilde{L}_n\tilde{v}_n+\tilde{K}\mathcal{W}_n[u^*+\phi^s_n] +\tilde{K}D\mathcal{W}_n[u^*+\phi^s_n]\tilde{v}_n+(\Kcrit+K^*_n+\tilde{K})R_n(\tilde{v}_n)
    \end{split}
    \end{equation}
By Lemma~\ref{lem:tildelambda}, the linearization $\tilde{L}_n$ has a simple eigenvalue near the origin (for $n$ sufficiently large) while the remainder of the spectrum is separated from the imaginary axis and contained in a ball lying strictly to the left of the line $\mathrm{Re}(\lambda)=-\alpha$. 

We therefore obtain an analogous result to that of Lemma~\ref{lem:DFprop} which we state now.

\begin{lem} \label{lem:DFnprop}
    There exists a $N\geq 1$ such that for all $n\geq N$ the linearization $\tilde{L}_n:X_n\to X_n$ has the following properties 
    \begin{itemize}
    \item[i)] The spectrum $\sigma(\tilde{L}_n)$ as an operator on $X_n$ has the decomposition $\sigma=\tilde{\sigma}_0\cup \tilde{\sigma_s}$ where $\tilde{\sigma}_0=\{\tilde{\lambda}_n\}$ and for $\alpha>0$ as in Lemma~\ref{lem:DFprop} it holds that 
    \[ 
        \sup_{\lambda\in \tilde{\sigma}_s} \mathrm{Re}(\lambda)<-\alpha. 
        \]
        \item[ii)] Restricted to the space $X_n$ the algebraic multiplicity of $\tilde{\lambda}_n$ as an eigenvalue of  $\tilde{L}_n$ is one and there exists a function $\tilde{v}_n^*(x)\in X_n$ normalized such that the (central) spectral projection has the following representation 
        \[ 
            \tilde{P}_n^cf=\tilde{v}_n^*(x)\int_0^1 f(y)\tilde{v}_n^*(y)\mathrm{d} y. 
        \]
        \item[iii)] For $\xi>0$ and $r<\xi-\alpha$ as in Lemma~\ref{lem:DFprop}, we have that  
        \[ 
            \tilde{\sigma}_s\subset \left\{ z\in \mathbb{C} \ | \ |z+\xi|< r \right\},  
        \]
       with (stable) spectral projection $\tilde{P}^s_n=I-\tilde{P}^c_n$.
       \item[iv)] Let $\tilde{X}^s_n=\mathrm{Rng} \tilde{P}^s$ and define $\tilde{L}^s_n=\tilde{L}_n|_{\tilde{X}^s_n}$.  Then $\tilde{L}^s_n$ generates an analytic semigroup on $\tilde{X}^s_n$ which we denote $e^{\tilde{L}_n^st}$.  Moreover, the following estimate holds
       \[ 
            \|e^{\tilde{L}_n^s t} \|\leq C e^{-\alpha t},
        \]
        for any $t>0$  and a constant $C$ independent of $n$.
       \item[v)]  Let $f\in C_{\tilde{\eta}}(\mathbb{R},\tilde{X}^s_n)$. Then 
       \[ 
        \frac{d\tilde{v}^s_n}{dt}=\tilde{L}_n^s\tilde{v}^s_n+f(t), 
       \]
       has a unique solution given by
       \[ 
        \tilde{v}^s_n(t)=\int_{-\infty}^t e^{(t-\tau)\tilde{L}^s_n} \tilde{P}_n^s f(\tau) \mathrm{d} \tau. 
        \]
       Furthermore, there exist a continuous function $\tilde{\kappa}(\tilde{\eta})$ such that $\|\tilde{v}^s_n\|_{C_{\tilde{\eta}}}\leq \tilde{\kappa}(\tilde{\eta})\|f\|_{C_{\tilde{\eta}}}$ for $\tilde{\eta}$ sufficiently small.
    \end{itemize}
\end{lem}

\begin{proof}
    Properties (i) through (iii) follow from the spectral results obtained in Lemma~\ref{lem:tildelambda}. The primary challenge is to validate that the constant $C$ in (iv) can be chosen independent of $n$, after which (v) follows from calculations analogous to those in Lemma~\ref{lem:DFprop}. As we did in Lemma~\ref{lem:DFprop}, we obtain a formula for the resolvent operator via Neumann series:
    \[ 
        \left(\tilde{L}^s_n-\lambda\right)^{-1}w=-\sum_{k=0}^\infty  \left(\frac{\tilde{L}^s_n+\xi}{\lambda+\xi}\right)^k \frac{w}{(\lambda+\xi)},
    \]
    where $\xi > 0$ is as given in property (iii). Since the spectrum of the shifted operator $\tilde{L}^s_n+\xi$ is contained strictly inside a ball of radius $r$ we therefore have the resolvent operator is bounded for any $|\lambda+\xi|>r$. Furthermore, there exists a $\varphi\in \left(\frac{\pi}{2},\pi \right)$
    such that for any $\lambda$ satisfying $|\mathrm{arg}(\lambda+\alpha)|<\varphi$ the following resolvent estimate holds:
        \[ 
        \left\|(\lambda+\alpha) \left(\tilde{L}^s_n-\lambda\right)^{-1}\right\|_{\infty \to \infty}\leq \left\| (\lambda+\alpha)\sum_{k=0}^\infty  \left(\frac{\tilde{L}^s_n+\xi}{\lambda+\xi}\right)^k \frac{1}{(\lambda+\xi)}\right\|_{\infty \to \infty}\leq \Xi_n.  
    \]
    We now verify that $\Xi_n$ can be taken independently of $n$.  Our strategy is as follows.  We show that the resolvent operator is uniformly bounded on $L^2$ and then use this to derive uniformity with respect to the norm on $X_n$.   
    
    The second resolvent identity; \[
    \left(\tilde{L}^s_n-\lambda\right)^{-1}-\left(L^s-\lambda\right)^{-1}=\left(\tilde{L}^s_n-\lambda\right)^{-1}\left(L^s-\tilde{L}^s_n\right)\left(L^s-\lambda\right)^{-1},\] implies that 
    \[ 
        \left(\tilde{L}^s_n-\lambda\right)^{-1} = \left[ I - (L^s-\tilde{L}^s_n)\left(L^s-\lambda\right)^{-1}\right]^{-1} \left(L^s-\lambda\right)^{-1}.
    \]
    Note that the inverse of the terms in the square brackets may be obtained from a Neumann series expansion provided that $\|\tilde{L}^s_n-L^s\|_{2\to 2}$ is sufficiently small. Therefore, when considered as an operator on $L^2([0,1])$, the operator norm convergence of $\tilde{L}^s_n\to L^s$ implies the following resolvent bound
    \begin{equation} \label{eq:L2resolvent} 
        \left\|\left(\tilde{L}^s_n-\lambda\right)^{-1}\right\|_{2\to 2} \leq \frac{\Theta}{|\lambda+\alpha|}, 
    \end{equation}
    for some constant $\Theta>0$ and independent of $n$. 
    
    In the case of $\tilde{X}_n^s$ equipped with the supremum norm we no longer have operator norm convergence in general; see \cite{bramburger2024persistence}. So, suppose for the sake of contradiction that 
    \[ 
        \left\|\left(\tilde{L}^s_n-\lambda\right)^{-1}\right\|_{\infty \to \infty} \leq \frac{\Xi_n}{|\lambda+\alpha|}, 
    \]
    but with $\Xi_n\to\infty$.  This would imply that there exists a sequence $w_n\in \tilde{X}^s_n$ with $\|w_n\|_\infty=1$ but for which $v_n=\left(\tilde{L}^s_n-\lambda\right)^{-1}w_n$ satisfies $\|v_n\|_\infty\to \infty$ as $n\to\infty$.   Since $\|w_n\|_2\leq \|w_n\|_\infty$ the resolvent estimate (\ref{eq:L2resolvent}) implies that $\|v_n\|_2$ is uniformly bounded.  We then argue as in the proof of Lemma~\ref{lem:specDFn}.  Recall that $\tilde{L}_n=(\Kcrit+K^*_n)D\mathcal{W}_n[u^*+\phi^s_n]$.  Let
    \[ \tilde{Q}_n(x)=-(\Kcrit+K^*_n)\int_0^1 W_n(x,y)\cos\left(u^*(y)+\phi^s_n(y)-u^*(x)-\phi^s_n(x) \right)\mathrm{d}y,\]
    be the multiplication part of the operator $\tilde{L}_n$.  Recall the definition of $Q_n(x)$ in (\ref{eq:LandLndecomp}).  We note that Lemma~\ref{lem:tildelambda} implies that $\|Q_n(x)-\tilde{Q}_n(x)\|_\infty \to 0$ as $n\to\infty$.  Therefore $\tilde{Q}_n(x)-\lambda\neq 0$ for $|\lambda+\xi|>r$ and $n$ sufficiently large.  
Recalling  that $\tilde{L}^s_n=\tilde{P}^s_n\tilde{L}_n$, we express
    \begin{equation} 
        \begin{split} 
            w_n&= (\tilde{Q}_n(x)-\lambda) v_n+(\Kcrit+K^*_n)\int_0^1 W_n(x,y)\cos(u^*(y)+\phi^s_n-u^*(x)-\phi^s_n(x))v_n(y)\mathrm{d}y \\
            &- \tilde{v}^*_n \left[ \int_0^1 \tilde{v}^*_n(y)\tilde{Q}_n(y)v_n(y)\mathrm{d}y  \right. \\ 
            &+ \left. (\Kcrit+K^*_n)\int_0^1 \int_0^1 \tilde{v}^*_n(x) W_n(x,y)\cos(u^*(y)+\phi^s_n-u^*(x)-\phi^s_n(x))v_n(y)\mathrm{d}y\mathrm{d}x \right].
        \end{split}
    \end{equation}
    This can be re-arranged as 
    \begin{equation} 
        \begin{split} 
            v_n(x) &= \frac{w_n(x)}{\tilde{Q}_n(x)-\lambda} -\frac{(\Kcrit+K^*_n)}{\tilde{Q}_n(x)-\lambda}\int_0^1 W_n(x,y)\cos(u^*(y)+\phi^s_n(y)-u^*(x)-u^*(y))v_n(y)\mathrm{d}y \\ 
            &+ \frac{ \tilde{v}^*_n}{\tilde{Q}_n(x)-\lambda} \left[ \int_0^1 \tilde{v}^*_n(y)\tilde{Q}_n(y)v_n(y)\mathrm{d}y  \right. \\
            &+ \left.(\Kcrit+K^*_n)\int_0^1 \int_0^1 \tilde{v}_n^*(x) W_n(x,y)\cos(u^*(y)+\phi^s_n(y)-u^*(x)-\phi^s_n(x))v_n(y)\mathrm{d}y\mathrm{d}x \right].
        \end{split}
    \end{equation}
    H\"older's inequality then implies that uniform boundedness of $v_n$ in $L^2$ translates to uniform boundedness of $v_n$ in $\tilde{X}^s_n$ equipped with the $L^\infty$ norm. The stated temporal bound in (iv) can then be obtained by standard estimates. Property (v) follows as in the proof of Lemma~\ref{lem:DFprop} and we omit the details.  This completes the proof. 
\end{proof}

In contrast to the construction of the center manifold in the graphon case, in this situation we no longer have a zero eigenvalue, but rather an isolated eigenvalue close to the origin.  To account for this, following \cite{haragus11}, we instead study the equation $\frac{d\tilde{v}_n}{dt}=J_n(\tilde{v}_n,\tilde{K},\nu)$ where 
\[ 
    J_n(\tilde{v}_n,\tilde{K},\nu)=M_n\tilde{v}_n+\tilde{K}\mathcal{W}_n[u^*+\phi^s_n]  +\nu \tilde{P}^c_n \tilde{v}_n+\tilde{K}D\mathcal{W}_n(u^*+\phi^s_n)\tilde{v}_n+(\Kcrit+K^*_n+\tilde{K})R_n(\tilde{v}_n),
\]
see \eqref{eq:HNexp} for reference but with the linear operator $\tilde{L}_n$ replaced with
\[ 
    M_n=\tilde{L}_n-\tilde{\lambda}_n\tilde{P}^c_n.
\]
The spectrum of $M_n$ on the space $X_n$ consists of an algebraically simple isolated eigenvalue at the origin with the rest of the spectrum being contained in the set $\mathrm{Re}(\lambda)\leq -\alpha$ for $n$ sufficiently large.  Therefore the equation $\frac{d\tilde{v}_n}{dt}=J_n(\tilde{v}_n,\tilde{K},\nu)$ satisfies the hypothesis of the center manifold theorem given in  \cite[Theorem 3.3]{haragus11}.  The manifold can be expressed as a graph
\[ 
    \mathcal{M}_{n,\nu}=\left\{ w^c_n \tilde{v}_n^*+\Psi_n\left(w^c_n,\tilde{K},\nu\right)\  | \  w^c_n\in \mathbb{R} \right\}, 
\]
where the function $\Psi_n:\mathbb{R}^3\to \tilde{X}^s_n$ is $C^k$ for any $k>2$. 

The reduced equation on the center manifold is obtained as
\begin{equation} \frac{dw^c_n}{dt} = \left\langle 
    (\nu-\tilde{\lambda}_n)\tilde{P}^c_n(w^c_n \tilde{v}_n^*)+\Omega_n+(\Kcrit +K^*_n+\tilde{K})\mathcal{W}_n\left[u^*+\phi^s_n+w^c_n \tilde{v}_n^*+\Psi_n\left(w^c_n,\tilde{K},\nu\right)\right],\tilde{v}^*_n\right\rangle.  \label{eq:stepgraphonprojectCM} 
\end{equation}
Let $h_{c,n}(w^c_n,\tilde{K},\nu)$ denote the right hand side of (\ref{eq:stepgraphonprojectCM}).  We note that $h_{c,n}(w^c_n,\tilde{K},\nu)=a_n\tilde{K}+\mathcal{O}(2)$ where $a_n=\langle \mathcal{W}_n[u^*+\phi^s_n],\tilde{v}^*_n\rangle$ and $\mathcal{O}(2)$ denotes terms that are at least quadratic in the variables $(w^c_n,\tilde{K},\nu)$. Then, the mapping $\Psi_n$ satisfies
\begin{equation} \label{eq:stepgraphonreductionequation}
    \left(D_{w^c_n}\Psi_n\right) h_{c,n}(w^c,\tilde{K},\nu)=\tilde{P}^s_n \left( \Omega_n(x)+(\Kcrit+K^*_n+\tilde{K})\mathcal{W}_n[u^*+\phi^s_n+w^c_n\tilde{v}^*_n+\Psi_n(w^c_n,\tilde{K},\nu)]\right).
\end{equation}
To obtain an expansion for $\Psi_n$ we begin with the linear ansatz:
\[ \Psi_n(w^c_n,\tilde{K},\nu)=\Psi_{n,100}w^c_n+\Psi_{n,010}\tilde{K}+\Psi_{n,001}\nu+\mathcal{O}(2),  \]
where $\Psi_{n,\cdot\cdot\cdot}$ are elements of $\tilde{X}^s_n$. Substituting into \eqref{eq:stepgraphonreductionequation} and retaining only linear terms we obtain the solvability condition  
\[ 
    a_n\Psi_{n,100}\tilde{K} = w^c_n \tilde{L}^s_n \Psi_{n,100}+\tilde{K} \tilde{L}^s_n \Psi_{n,010}+\nu \tilde{L^s_n}\Psi_{n,001} +\tilde{K} \tilde{P}^s_n \mathcal{W}_n[u^*+\phi^s_n],   
\]
from which we obtain $\Psi_{n,100}=\Psi_{n,001}=0$ while
\[
    \Psi_{n,010}=-(\tilde{L}^s_n)^{-1}\tilde{P}^s_n\mathcal{W}_n[u^*+\phi^s_n].  
\]
Since $\Psi_n$ lacks linear terms in $w^c_n$ and $R_n$ is quadratic in its argument we obtain that the reduced equation on the center manifold has the expansion
\[ \frac{dw^c_n}{dt}=\nu w^c_n+a_n\tilde{K}+b_n(w^c_n)^2 +\mathcal{O}\left(w^c_n\tilde{K},\tilde{K}^2,|w^c+\tilde{K}+\nu|^3 \right),\]
where $a_n$ and $b_n$ are given in (\ref{eq:anbns}).

The proof of the center manifold theorem in \cite{haragus11} requires the use of a cut-off function that then describes the size of the neighborhood on which the center manifold reduction is valid. In what follows we establish uniformity in large $n$ of the size of this neighborhood of validity. 

Following \cite{haragus11} let $\mathcal{V}_n=\left( \tilde{v}_n , \tilde{K} , \nu\right)^T$ and 
\[ 
    \mathcal{L}_n=\left(\begin{array}{ccc} M_n & \mathcal{W}_n[u^*] & \langle \cdot, \tilde{v}^*_n \rangle \\ 0 & 0 & 0 \\ 0 & 0 & 0 \end{array}\right), \ \mathcal{N}_n(\mathcal{V}_n)=\left(\begin{array}{c} \tilde{K}D\mathcal{W}_n[u^*+\phi^s_n]\tilde{v}_n+(\Kcrit+K^*_n+\tilde{K})R_n(\tilde{v}_n) \\ 0 \\ 0 \end{array}\right),
\]
so that the system is recast as 
\[ 
    \partial_t\mathcal{V}_n=\mathcal{L}_n\mathcal{V}_n+\mathcal{N}_n(\mathcal{V}_n). 
\]
The spectrum of $\mathcal{L}_n$ is unchanged while the algebraic multiplicity of zero is now three.  There exist center and stable projections associated to these spectral sets which we denote $\mathcal{P}_n^c$ and $\mathcal{P}_n^s$. Note that these projections have the following structure
\[ \mathcal{P}_n^c=\left(\begin{array}{ccc} \tilde{P}^c_n & * & 0 \\ 0 & 1& 0 \\ 0 & 0 & 1 \end{array}\right), \quad \mathcal{P}_n^s=\left(\begin{array}{ccc} \tilde{P}^s_n & * & 0 \\ 0 & 0& 0 \\ 0 & 0 & 0 \end{array}\right),\]
where the $*$ denote  non-zero terms which will not be relevant to the remaining analysis.  
A smooth cut-off function is selected to modify the nonlinearity.  The following modified nonlinearity is considered 
\[ 
    \mathcal{N}_n^\varepsilon(\tilde{v}_n,\tilde{K})=\chi\left( \frac{\|(w^c_n,\tilde{K},\nu)\|_\infty}{\varepsilon}\right) \mathcal{N}_n(\tilde{v}_n,\tilde{K}),
\]
where $\chi:\mathbb{R}\to\mathbb{R}$ is a smooth cutoff function satisfying $\chi(x)=0$ if $|x|\leq 1$ and $\chi(x)=0$ if $|x|\geq 2$. This gives that the system is unchanged when $|w^c|\leq \varepsilon$, $|\tilde{K}|\leq \varepsilon$, and $|\nu|\leq \varepsilon$, i.e. $\mathcal{N}_n^\varepsilon(\tilde{v}_n,\tilde{K})=\mathcal{N}(\tilde{v}_n,\tilde{K})$ whenever $|\langle \tilde{v}_n,\tilde{v}_n^*\rangle|\leq \varepsilon$ and $|\tilde{K}|<\varepsilon$. Then, the contraction mapping employed in the proof of the center manifold requires control of three terms: the function $\kappa_n(\eta)$ appearing in Lemma~\ref{lem:DFnprop} and the quantities 
\begin{equation}
    \begin{split}
        \delta_{0,n}(\varepsilon) &= \sup_{\tilde{K}\in \mathbb{R},\nu\in\mathbb{R},v\in \tilde{X}_n^c\times B_\varepsilon(0)\subset \tilde{X}_n^s} \left\{ \| \mathcal{P}_n^c \mathcal{N}_n^\varepsilon(v,\tilde{K}) \|_{\infty},   \| \mathcal{P}_n^s \mathcal{N}_n^\varepsilon(v,\tilde{K}) \|_{\infty} \right\} \\
         \delta_{1,n}(\varepsilon) &= \sup_{\tilde{K}\in \mathbb{R},\nu\in\mathbb{R},v\in \tilde{X}_n^c\times B_\varepsilon(0)\subset \tilde{X}_n^s}  \left\{ \| D_{\mathcal{V}}\mathcal{P}_n^c \mathcal{N}_n^\varepsilon(v,\tilde{K}) \|_{\infty\to\infty},   \| D_{\mathcal{V}} \mathcal{P}_n^s \mathcal{N}_n^\varepsilon(v,\tilde{K}) \|_{\infty\to \infty} \right\}.
    \end{split}
\end{equation}
We now provide the following estimates.

\begin{lem}
    Under the assumptions of Lemma~\ref{lem:DFnprop}, there exist positive constants $C_0$ and $C_1$, independent of $n$, such that 
    \begin{equation} 
        \delta_{0,n}(\varepsilon)\leq C_0 \varepsilon^2, \quad \delta_{1,n}(\varepsilon)\leq C_1 \varepsilon,\label{eq:deltasestimate} 
    \end{equation}
    for any $\varepsilon>0$
\end{lem}
\begin{proof} The $\varepsilon$-scaling in both estimates stems from the quadratic nature of the nonlinearity $\mathcal{N}_n(v,\tilde{K})$ and smoothness of the cut-off function, so the main item to prove is that the scaling constants $C_0$ and $C_1$ may be chosen independent of $n$.

To condense notation, let
\[ 
    Q(x,y)=u^*(y)+\phi^s_n(y)-u^*(x)-\phi^s_n(x), \quad \Delta_v(x,y)=\tilde{v}_n(y)-\tilde{v}_n(x).
\]
Then
\begin{equation}
    R_n(\tilde{v}_n)=\int_0^1 W_n(x,y)\left[ \sin\left(Q(x,y)+\Delta_v(x,y)\right)-\cos(Q(x,y))\Delta_v(x,y)-\sin(Q(x,y))\right] \mathrm{d}y.
\end{equation}
Combining the fact that $0\leq W_n(x,y)\leq 1$ with Taylor's Theorem gives
$\|\tilde{v}_n\|^2_\infty$ only 
\[ 
    \|R_n(\tilde{v}_n)\|_\infty \leq \frac{1}{2}\sup_{(x,y)\in [0,1]^2}\left|\Delta_v(x,y)^2\right|\leq  \|\tilde{v}_n\|^2_\infty.
\] 
We therefore obtain
\[ 
    \|\mathcal{N}^{(1)}_n(\tilde{v}_n,\tilde{K})\|_\infty \leq \|D\mathcal{W}_n[u^*+\phi^s_n]\|_{\infty\to\infty} |\tilde{K}|\|\tilde{v}_n\|_\infty+ |\Kcrit+K^*_n+\tilde{K}|\|\tilde{v}_n\|_\infty^2,  
\]
where $\mathcal{N}^{(1)}_n$ denotes the first, and only non-zero, entry of $\mathcal{N}_n$.
We have a coarse bound for the operator norm $\|D\mathcal{W}_n[u^*+\phi^s_n]\|_{\infty\to\infty}\leq 2 $ after recalling that
\[ 
    D\mathcal{W}_n[u^*+\phi^s_n]w=\int_0^1 W_n(x,y)\cos\left( u^*(y)+\phi^s_n(y)-u^*(x)-\phi^s_n(x)\right)(w(y)-w(x))\mathrm{d}y 
\]
and provided $|\tilde{K}|\leq \Kcrit/2$ we also obtain $|\Kcrit+K^*_n+\tilde{K}|\leq 2\Kcrit$ for $n$ sufficiently large since we have shown in Lemma~\ref{lem:maptozero} that $K^*_n\to 0$.  This provides an $n$-independent bound 
\[ 
    \|\mathcal{N}^{(1)}_n(\tilde{v}_n,\tilde{K})\|_\infty \leq 2 |\tilde{K}|\|\tilde{v}_n\|_\infty+ 2\Kcrit\|\tilde{v}_n\|_\infty^2. 
\]
Then owing to the structure of the center projection $\mathcal{P}^c_n$ we have that 
\[ 
    \mathcal{P}^c_n\mathcal{N}_n(\tilde{v}_n,\tilde{K})= \left(\begin{array}{c}\tilde{P}^c_n \mathcal{N}^{(1)}_n \\ 0 \\ 0\end{array}\right), 
\]
and therefore we obtain
\[ 
    \|\mathcal{P}^c_n\mathcal{N}^\varepsilon_n(\tilde{v}_n,\tilde{K})\|_\infty \leq 2(1+\Kcrit)\|\tilde{v}^*_n\|^2_\infty \varepsilon^2, 
\]
where the additional $\|\tilde{v}^*_n\|^2_\infty$ comes from the application of $\tilde{P}^c_n$ to $\mathcal{N}^{(1)}_n$ through the results of Lemma~\ref{lem:DFnprop}. Similarly, the stable projection yields the stated estimate for $\delta_{0,n}(\varepsilon)$ in \eqref{eq:deltasestimate}.

The verification that $C_1$ may be chosen independently of $n$ follows from a similar line of analysis so we omit the details.  


\end{proof}

We now have the existence of a locally invariant center manifold depending on the artificial parameter $\nu$.  Taking $\nu=\tilde{\lambda}_n$ for $n$ sufficiently large so that $|\tilde{\lambda}_n|<\varepsilon$, we then recover the result stated in Proposition~\ref{prop:stepgraphonCM}.

\end{appendix}

\bibliographystyle{abbrv}
\bibliography{references.bib}

@article {borgs11,
    AUTHOR = {Borgs, Christian and Chayes, Jennifer and Lov\'asz,
              L\'aszl\'o{} and S\'os, Vera and Vesztergombi, Katalin},
     TITLE = {Limits of randomly grown graph sequences},
   JOURNAL = {European J. Combin.},
  FJOURNAL = {European Journal of Combinatorics},
    VOLUME = {32},
      YEAR = {2011},
    NUMBER = {7},
     PAGES = {985--999},
      ISSN = {0195-6698,1095-9971},
   MRCLASS = {05C80 (60F15)},
  MRNUMBER = {2825531},
MRREVIEWER = {Tatyana\ S.\ Turova},
       DOI = {10.1016/j.ejc.2011.03.015},
       URL = {https://doi.org/10.1016/j.ejc.2011.03.015},
}

@article{bramburger2023pattern,
  title={Pattern formation in random networks using graphons},
  author={Bramburger, Jason and Holzer, Matt},
  journal={SIAM Journal on Mathematical Analysis},
  volume={55},
  number={3},
  pages={2150--2185},
  year={2023},
  publisher={SIAM}
}

@article{vizuete2021laplacian,
  title={The {L}aplacian spectrum of large graphs sampled from graphons},
  author={Vizuete, Renato and Garin, Federica and Frasca, Paolo},
  journal={IEEE Transactions on Network Science and Engineering},
  volume={8},
  number={2},
  pages={1711--1721},
  year={2021},
  publisher={IEEE}
}

@article{garin2024corrections,
  title={Corrections to and improvements on results from ``{T}he {L}aplacian spectrum of large graphs sampled from graphons''},
  author={Garin, Federica and Frasca, Paolo and Vizuete, Renato},
  journal={arXiv preprint arXiv:2407.14422},
  year={2024}
}

@article{dorfler2011critical,
  title={On the critical coupling for {K}uramoto oscillators},
  author={D{\"o}rfler, Florian and Bullo, Francesco},
  journal={SIAM Journal on Applied Dynamical Systems},
  volume={10},
  number={3},
  pages={1070--1099},
  year={2011},
  publisher={SIAM}
}

@article{ermentrout1985synchronization,
  title={Synchronization in a pool of mutually coupled oscillators with random frequencies},
  author={Ermentrout, G Bard},
  journal={Journal of Mathematical Biology},
  volume={22},
  number={1},
  pages={1--9},
  year={1985},
  publisher={Springer}
}

@inproceedings{kuramoto1975self,
  title={Self-entrainment of a population of coupled non-linear oscillators},
  author={Kuramoto, Yoshiki},
  booktitle={International Symposium on Mathematical Problems in Theoretical Physics: January 23--29, 1975, Kyoto University, Kyoto/Japan},
  pages={420--422},
  year={1975},
  organization={Springer}
}

@article{hammond2007pathological,
  title={Pathological synchronization in Parkinson's disease: networks, models and treatments},
  author={Hammond, Constance and Bergman, Hagai and Brown, Peter},
  journal={Trends in neurosciences},
  volume={30},
  number={7},
  pages={357--364},
  year={2007},
  publisher={Elsevier}
}

@article{nagpal2024synchronization,
  title={Synchronization in random networks of identical phase oscillators: A graphon approach},
  author={Nagpal, Shriya V and Nair, Gokul G and Strogatz, Steven H and Parise, Francesca},
  journal={arXiv preprint arXiv:2403.13998},
  year={2024}
}

@article{medvedev2014nonlinear,
  title={The nonlinear heat equation on {W}-random graphs},
  author={Medvedev, Georgi S},
  journal={Archive for Rational Mechanics and Analysis},
  volume={212},
  pages={781--803},
  year={2014},
  publisher={Springer}
}

@article{ling2020critical,
  title={On the critical coupling of the finite {K}uramoto model on dense networks},
  author={Ling, Shuyang},
  journal={arXiv preprint arXiv:2004.03202},
  year={2020}
}

@article{lehnertz2009synchronization,
  title={Synchronization phenomena in human epileptic brain networks},
  author={Lehnertz, Klaus and Bialonski, Stephan and Horstmann, Marie-Therese and Krug, Dieter and Rothkegel, Alexander and Staniek, Matth{\"a}us and Wagner, Tobias},
  journal={Journal of neuroscience methods},
  volume={183},
  number={1},
  pages={42--48},
  year={2009},
  publisher={Elsevier}
}

@article{rohden2012self,
  title={Self-organized synchronization in decentralized power grids},
  author={Rohden, Martin and Sorge, Andreas and Timme, Marc and Witthaut, Dirk},
  journal={Physical review letters},
  volume={109},
  number={6},
  pages={064101},
  year={2012},
  publisher={APS}
}

@article{dorfler2013synchronization,
  title={Synchronization in complex oscillator networks and smart grids},
  author={D{\"o}rfler, Florian and Chertkov, Michael and Bullo, Francesco},
  journal={Proceedings of the National Academy of Sciences},
  volume={110},
  number={6},
  pages={2005--2010},
  year={2013},
  publisher={National Acad Sciences}
}

@article{motter2013spontaneous,
  title={Spontaneous synchrony in power-grid networks},
  author={Motter, Adilson E and Myers, Seth A and Anghel, Marian and Nishikawa, Takashi},
  journal={Nature Physics},
  volume={9},
  number={3},
  pages={191--197},
  year={2013},
  publisher={Nature Publishing Group UK London}
}

@article{singer1999neuronal,
  title={Neuronal synchrony: a versatile code for the definition of relations?},
  author={Singer, Wolf},
  journal={Neuron},
  volume={24},
  number={1},
  pages={49--65},
  year={1999},
  publisher={Elsevier}
}

@article{uhlhaas2009neural,
  title={Neural synchrony in cortical networks: history, concept and current status},
  author={Uhlhaas, Peter and Pipa, Gordon and Lima, Bruss and Melloni, Lucia and Neuenschwander, Sergio and Nikoli{\'c}, Danko and Singer, Wolf},
  journal={Frontiers in integrative neuroscience},
  volume={3},
  pages={543},
  year={2009},
  publisher={Frontiers}
}

@article{jutras2010synchronous,
  title={Synchronous neural activity and memory formation},
  author={Jutras, Michael J and Buffalo, Elizabeth A},
  journal={Current opinion in neurobiology},
  volume={20},
  number={2},
  pages={150--155},
  year={2010},
  publisher={Elsevier}
}

@article{axmacher2006memory,
  title={Memory formation by neuronal synchronization},
  author={Axmacher, Nikolai and Mormann, Florian and Fern{\'a}ndez, Guillen and Elger, Christian E and Fell, Juergen},
  journal={Brain research reviews},
  volume={52},
  number={1},
  pages={170--182},
  year={2006},
  publisher={Elsevier}
}

@article{bick2024dynamical,
  title={Dynamical Systems on Graph Limits and Their Symmetries},
  author={Bick, Christian and Sclosa, Davide},
  journal={Journal of Dynamics and Differential Equations},
  pages={1--36},
  year={2024},
  publisher={Springer}
}

@article{bick2020understanding,
  title={Understanding the dynamics of biological and neural oscillator networks through exact mean-field reductions: a review},
  author={Bick, Christian and Goodfellow, Marc and Laing, Carlo R and Martens, Erik A},
  journal={The Journal of Mathematical Neuroscience},
  volume={10},
  number={1},
  pages={9},
  year={2020},
  publisher={Springer}
}

@book{kuramoto1984chemical,
  title={Chemical Oscillations, Waves, and Turbulence},
  author={Kuramoto, Yoshiki},
  year={1984},
  publisher={Springer}
}

@book {janson2010graphons,
    AUTHOR = {Janson, Svante},
     TITLE = {Graphons, cut norm and distance, couplings and rearrangements},
    SERIES = {New York Journal of Mathematics. NYJM Monographs},
    VOLUME = {4},
 PUBLISHER = {State University of New York, University at Albany, Albany,
              NY},
      YEAR = {2013},
     PAGES = {76},
   MRCLASS = {05C80 (28A20)},
  MRNUMBER = {3043217},
MRREVIEWER = {David B. Penman},
}

@article{chiba2018mean,
  title={The mean field analysis of the {K}uramoto model on graphs {I}. The mean field equation and transition point formulas},
  author={Chiba, Hayato and Medvedev, Georgi S},
  journal={Discrete and Continuous Dynamical Systems},
  volume={39},
  number={1},
  pages={131--155},
  year={2019},
  publisher={Discrete and Continuous Dynamical Systems}
}

@article{chiba2018bifurcations,
  title={Bifurcations in the {K}uramoto model on graphs},
  author={Chiba, Hayato and Medvedev, Georgi S and Mizuhara, Matthew S},
  journal={Chaos: An Interdisciplinary Journal of Nonlinear Science},
  volume={28},
  number={7},
  year={2018},
  publisher={AIP Publishing}
}

@article{acebron2005kuramoto,
  title={The {K}uramoto model: A simple paradigm for synchronization phenomena},
  author={Acebr{\'o}n, Juan A and Bonilla, Luis L and Vicente, Conrad J P{\'e}rez and Ritort, F{\'e}lix and Spigler, Renato},
  journal={Reviews of modern physics},
  volume={77},
  number={1},
  pages={137},
  year={2005},
  publisher={APS}
}

@article{strogatz2000kuramoto,
  title={From {K}uramoto to {C}rawford: exploring the onset of synchronization in populations of coupled oscillators},
  author={Strogatz, Steven H},
  journal={Physica D: Nonlinear Phenomena},
  volume={143},
  number={1-4},
  pages={1--20},
  year={2000},
  publisher={Elsevier}
}

@article{ling2019landscape,
  title={On the landscape of synchronization networks: A perspective from nonconvex optimization},
  author={Ling, Shuyang and Xu, Ruitu and Bandeira, Afonso S},
  journal={SIAM Journal on Optimization},
  volume={29},
  number={3},
  pages={1879--1907},
  year={2019},
  publisher={SIAM}
}

@article{kassabov2022global,
  title={A global synchronization theorem for oscillators on a random graph},
  author={Kassabov, Martin and Strogatz, Steven H and Townsend, Alex},
  journal={Chaos: An Interdisciplinary Journal of Nonlinear Science},
  volume={32},
  number={9},
  year={2022},
  publisher={AIP Publishing}
}

@article{medvedev2014small,
  title={Small-world networks of {K}uramoto oscillators},
  author={Medvedev, Georgi S},
  journal={Physica D: Nonlinear Phenomena},
  volume={266},
  pages={13--22},
  year={2014},
  publisher={Elsevier}
}

@book {lovasz12,
    AUTHOR = {Lov\'{a}sz, L\'{a}szl\'{o}},
     TITLE = {Large networks and graph limits},
    SERIES = {American Mathematical Society Colloquium Publications},
    VOLUME = {60},
 PUBLISHER = {American Mathematical Society, Providence, RI},
      YEAR = {2012},
     PAGES = {xiv+475},
      ISBN = {978-0-8218-9085-1},
   MRCLASS = {05-02 (05C60 05C80 05C82 05D40)},
  MRNUMBER = {3012035},
MRREVIEWER = {Anant P. Godbole},
       DOI = {10.1090/coll/060},
       URL = {https://doi.org/10.1090/coll/060},
}

@article{abdalla2022expander,
  title={Expander graphs are globally synchronising},
  author={Abdalla, Pedro and Bandeira, Afonso S and Kassabov, Martin and Souza, Victor and Strogatz, Steven H and Townsend, Alex},
  journal={arXiv preprint arXiv:2210.12788},
  year={2022}
}

@article{bramburger2024persistence,
  title={Persistence of steady-states for dynamical systems on large networks},
  author={Bramburger, Jason J and Holzer, Matt and Williams, Jackson},
  journal={arXiv preprint arXiv:2402.09276},
  year={2024}
}

@book{haragus11,
  title={Local bifurcations, center manifolds, and normal forms in infinite-dimensional dynamical systems},
  author={Haragus, Mariana and Iooss, G{\'e}rard},
  volume={3},
  year={2011},
  publisher={Springer}
}

@article{chiba15,
  title={A proof of the {K}uramoto conjecture for a bifurcation structure of the infinite-dimensional {K}uramoto model},
  author={Chiba, Hayato},
  journal={Ergodic Theory and Dynamical Systems},
  volume={35},
  number={3},
  pages={762--834},
  year={2015},
  publisher={Cambridge University Press}
}

@article{sakaguchi88,
    author = {Sakaguchi, Hidetsugu},
    title = {Cooperative Phenomena in Coupled Oscillator Systems under External Fields},
    journal = {Progress of Theoretical Physics},
    volume = {79},
    number = {1},
    pages = {39-46},
    year = {1988},
    month = {01},
    issn = {0033-068X},
    doi = {10.1143/PTP.79.39},
    url = {https://doi.org/10.1143/PTP.79.39},
    eprint = {https://academic.oup.com/ptp/article-pdf/79/1/39/6867039/79-1-39.pdf},
}

@article{ott08,
  title={Low dimensional behavior of large systems of globally coupled oscillators},
  author={Ott, Edward and Antonsen, Thomas M},
  journal={Chaos: An Interdisciplinary Journal of Nonlinear Science},
  volume={18},
  number={3},
  year={2008},
  publisher={AIP Publishing}
}

@book{billingsley,
  author    = {Patrick Billingsley},
  title     = {Probability and Measure},
  edition   = {3},
  year      = {1995},
  publisher = {John Wiley \& Sons},
  address   = {New York},
  isbn      = {978-0471007104}
}

@article {dietert18,
  title={The mathematics of asymptotic stability in the Kuramoto model},
  author={Dietert, Helge and Fernandez, Bastien},
  journal={Proceedings of the Royal Society A},
  volume={474},
  number={2220},
  pages={20180467},
  year={2018},
  publisher={The Royal Society Publishing}
}

@article {mirollo05,
    AUTHOR = {Mirollo, Renato E. and Strogatz, Steven H.},
     TITLE = {The spectrum of the locked state for the {K}uramoto model of
              coupled oscillators},
   JOURNAL = {Phys. D},
  FJOURNAL = {Physica D. Nonlinear Phenomena},
    VOLUME = {205},
      YEAR = {2005},
    NUMBER = {1-4},
     PAGES = {249--266},
      ISSN = {0167-2789,1872-8022},
   MRCLASS = {34C15 (34L30 92B05)},
  MRNUMBER = {2167156},
MRREVIEWER = {Jing\ Hui},
       DOI = {10.1016/j.physd.2005.01.017},
       URL = {https://doi.org/10.1016/j.physd.2005.01.017},
}

@article {mirollo07,
    AUTHOR = {Mirollo, R. and Strogatz, S. H.},
     TITLE = {The spectrum of the partially locked state for the {K}uramoto
              model},
   JOURNAL = {J. Nonlinear Sci.},
  FJOURNAL = {Journal of Nonlinear Science},
    VOLUME = {17},
      YEAR = {2007},
    NUMBER = {4},
     PAGES = {309--347},
      ISSN = {0938-8974,1432-1467},
   MRCLASS = {37C75 (34C15 37L60 37N25)},
  MRNUMBER = {2335124},
MRREVIEWER = {Miaohua\ Jiang},
       DOI = {10.1007/s00332-006-0806-x},
       URL = {https://doi.org/10.1007/s00332-006-0806-x},
}

@incollection {vanderbauwhede92,
  title={Center manifold theory in infinite dimensions},
  author={Vanderbauwhede, Andr{\'e} and Iooss, G},
  booktitle={Dynamics reported: expositions in dynamical systems},
  pages={125--163},
  year={1992},
  publisher={Springer}
}

@book {kato,
    AUTHOR = {Kato, Tosio},
     TITLE = {Perturbation theory for linear operators},
    SERIES = {Classics in Mathematics},
      NOTE = {Reprint of the 1980 edition},
 PUBLISHER = {Springer-Verlag, Berlin},
      YEAR = {1995},
     PAGES = {xxii+619},
      ISBN = {3-540-58661-X},
   MRCLASS = {47A55 (46-00 47-00)},
  MRNUMBER = {1335452},
}

\end{document}